\documentclass{amsart}
\usepackage{graphicx} % Required for inserting images
\usepackage{amsfonts, amsmath, amssymb, amsthm}
\usepackage{xcolor}
\usepackage{mathrsfs}
\usepackage{tcolorbox}
\usepackage{hyperref}
\usepackage{stix}
\usepackage{soul}
\usepackage[headheight=125pt,top=1.4in, bottom=1.4in, left=1.2in, right=1.2in]{geometry}
\title[Characterization of Fourier transform]{Characterization of Fourier transform of Hilbert space Valued functions}
\author[Achar, Biswas, and Manna]{Satyajyoti Achar, Md Hasan Ali Biswas and Ramesh Manna}
\address{School of Mathematical Sciences, National Institute of Science Education and
Research, Bhubaneswar, An OCC of Homi Bhabha National Institute, Jatni 752050, India.}
\email{satyajyoti.achar@niser.ac.in} 
\address{Department of Mathematics, Indian Institute of Science, Bangalore - 560012, India}
\email{mdhasanalibiswas4@gmail.com}
\address{School of Mathematical Sciences, National Institute of Science Education and
Research, Bhubaneswar, An OCC of Homi Bhabha National Institute, Jatni 752050, India.}
\email{rameshmanna@niser.ac.in}

\date{\today}

\date{\today}
\begin{document}
\newcommand{\R}{\mathbb{R}}
\newcommand{\Hsn}{\mathbb{H}}
\newcommand{\Sl}{\mathcal{L}}
\newcommand{\A}{\mathcal{A}}
\newcommand{\F}{\mathscr{F}}
\newcommand{\N}{\mathbb{N}}
\newcommand{\Hi}{\mathcal{H}}
\newcommand*{\bigchi}{\mbox{\Large$\chi$}}
\newcommand{\Hm}{\mathrm{H}}
\newcommand{\blp}{\big (}
\newcommand{\brp}{\big )}
\newcommand{\Z}{\mathbb{Z}}
\newcommand{\Za}{\mathscr{Z}}
\newcommand{\T}{\mathbb{T}}
\newcommand{\J}{\mathbf{J}}
\newcommand{\C}{\mathbb{C}}
\newcommand{\diff}{\mathrm{d}}
\newcommand{\Df}{{\mathcal{F}_T}}
\newcommand{\D}{\mathcal{D}}
\newcommand{\B}{\mathcal{B}}
\newcommand{\M}{\mathcal{M}}
\newcommand{\G}{\mathcal{G}}
\newcommand{\nc}{\circledast_A}
\newcommand{\tc}{\otimes_\theta^\lambda}
\newcommand{\im}{\mathrm{Im}}
\newcommand{\re}{\mathrm{Re}}
\newcommand{\ch}{\mathcal{C_\theta}}
\newcommand{\nch}{\mathcal{C_{-\theta}}}
\newcommand{\f}{\mathcal{F}}
\def\Msp{{\boldsymbol M}}
\newtheorem{theorem}{Theorem}[section]
\newtheorem{lemma}[theorem]{Lemma}

\theoremstyle{definition}
\newtheorem{definition}[theorem]{Definition}
\newtheorem{example}[theorem]{Example}
\newtheorem{OP}[theorem]{Open problem}
\newtheorem{AP}[theorem]{Appendix}
\newtheorem{corollary}[theorem]{Corollary}
\newtheorem{xca}[theorem]{Exercise}
\newtheorem{propn}[theorem]{Proposition}
\theoremstyle{remark}
\newtheorem{remark}[theorem]{Remark}

\numberwithin{equation}{section}

\begin{abstract}
    The aim of this paper is to obtain the characterization of Fourier transform for functions taking values in a separable Hilbert space $H$, defined on the groups $G=\R^n,\,G=\Z/n\Z$. In the first part of the paper, we consider the Fourier transform as a bijective mapping on $S_H$, the space of all $H$-valued smooth rapidly decaying functions on $\R^n$, which naturally extends bijectively to the dual space $S_H'$. Within this framework, we characterize the Fourier transform in terms of the convolution theorem, supplemented by additional structural properties of the Fourier transform. In the second part of the paper, we consider the discrete setting, where the Fourier transform is an operator on $L^1(\Z/n\Z, H)$, and we provide corresponding characterization based on its essential properties. 
    In contrast to the classical setting, a key assumption in our framework involves the interplay between two distinct convolution operations. As a consequence, the structure of our proof diverges notably from the classical case and introduces additional complexities.
\end{abstract}
\maketitle

\section{Introduction and Preliminaries}
Fourier transform is the central object in Harmonic Analysis. Its profound influence extends across numerous domains, including differential equations, number theory, and quantum mechanics, where its applications continue to shape and inspire new developments. There are a group of arguments by many mathematicians to explain why the Fourier transform is so important. 
The reason is that the Fourier transform is the only operator that meets the specified set of properties. This approach is referred to as ``Axiomatic'', see \cite{krukowski}. The Fourier transform possesses a rich collection of remarkable properties. This raises a natural and intriguing question: Are there certain properties that determine the Fourier transform uniquely among all operators?

\par It is well known that the Fourier transform $\F$ is a bijection on $S(\R^n)$, the Schwartz space of rapidly decreasing functions on $\R^n$, and it has a bijective extension on $S^\prime(\R^n)$, the space of all tempered distributions on $\R^n$. We refer to \cite{gel1964fand} for a study of the theory of distributions on $\R^n$. It is well known that 
\begin{align}
\label{eq: conv pointwise}\F(f\ast g)&=\F(f)\cdot\F(g)\\
\label{eq: pointwise conv}\F(f\cdot g)&=\F(f)\ast \F(g),
\end{align}
holds for all $f,g\in S(\R^n)$, where $f\ast g(x)=\int_{\R^n} f(x-y)g(y)dy$. The authors in \cite{Alesker_paper} (see also \cite{Alesker_BC}) proved that the bijectivity of the Fourier transform on $S(\R^n)$, $S^\prime(\R^n)$ and \eqref{eq: conv pointwise}, \eqref{eq: pointwise conv}, determines the Fourier transform upto some extent. We provide the precise statement below.
\begin{theorem}[\cite{Alesker_paper, Alesker_BC}]\label{th: alesker}
 Let $T:S(\R^n)\to S(\R^n)$ be a bijective mapping which admits a bijective extension $T^\prime:S^\prime(\R^n)\to S^\prime(\R^n)$ such that
 \begin{align*}
     T^\prime(f\ast g)&=T^\prime(f)\cdot T(g),~~\, \text{ for all }f\in S^\prime(\R^n), g\in S(\R^n),\\
      T^\prime(f\cdot g)&=T^\prime(f)\ast T(g),~~\, \text{ for all }f\in S^\prime(\R^n), g\in S(\R^n).
 \end{align*}
 Then, either $T(f)=\F(f\circ B)$ or $T(f)=\overline{\F(f\circ B)}$, holds for all $f\in S(\R^n)$, for some $B\in GL_n(\R)$ with $|det(B)|=1$.
\end{theorem}
The above theorem is obtained as a consequence of the following
    \begin{theorem}[\cite{Alesker_paper,Alesker_BC}]\label{th: 1.2}
    Let $U:S(\R^n)\to S(\R^n)$ be a bijective mapping having a bijective extension $U^\prime:S'(\R^n)\to S'(\R^n) $ such that
    \begin{align*}
    U'(f\ast g)=U'(f)\ast U(g)~~~\text{ and }~ U'(f\cdot g)&=U'(f)\cdot U(g),  
\end{align*}     
for all $f\in S'(\R^n),\,g\in S(\R^n)$. Then, either $ T(f)=f\circ B$ or $ T(f)= \overline{f\circ B}$, for some $B\in GL_n(\R)$ with $|det(B)|=1$. 
    \end{theorem}
It is important to note that, in the above theorem, no assumption are made regarding the linearity and continuity of the mapping $T$. Here, we would like to mention that in the proof of Theorem \ref{th: 1.2}, the authors made use of the action of $T$ on the elementary distributions such as the constant functions $0, 1$ and the Kronecker delta distributions.
\par 
Subsequently in \cite{alesker2010characterization}, the authors have shown that a bijective mapping $T$ on $S(\R^n)$ having a bijective extension $T'$ on $S'(\R^n)$ satisfying $T'(f\cdot g)=T'(f)\cdot T(g)$, for all $f\in S'(\R^n),\,g\in S(\R^n)$ is of the form $Tf=f\circ u$ or $Tf=\overline{f\circ u}$, where $u$ is a $C^\infty$-diffeomorphism on $\R^n$. As a consequence of the form of multiplication preserving operators, the authors obtained a characterization of the Fourier transform on $S(\R^n)$. Later in \cite{artstein2012multiplicative}, similar results were proved for functions defined on a manifold.

\par It is easy to see that the Fourier transform maps $L^1(G)$ into $C_0(G)$ boundedly, where $G$ is a locally compact abelian group. Further, appealing to Fubini's theorem, it is easy to see that the Fourier transform satisfies the equality \eqref{eq: conv pointwise} for $f,g\in L^1(\R^n)$, more generally for $f,g\in L^1(G)$, where $G$ is a locally compact abelian group. In \cite{Jaming}, Jaming proved that the boubdedness of the Fourier transform along with \eqref{eq: conv pointwise} for $f,g\in L^1(G)$ characterizes the Fourier transform to some extent when $G$ is taken to  be $\R^n$ or $\mathbb{T}^n$. We state his result for the case $G=\R^n$.

\begin{theorem}[\cite{Jaming}]\label{th: jaming}
    Let $T:L^1(\R^n)\to C_0(\R^n)$ be a bounded linear operator such that
    \[
    T(f\ast g)=T(f)\cdot T(g),~~\text{ for all } f,g\in L^1(\R^n).
    \]
    Then there exists $E\subset \R^n$ and $\phi:\R^n\to\R^n$ such that
    \[
    T(f)(\xi)=\chi_E(\xi)\F(f)(\phi(\xi)), \,\, \xi\in\R^n,
    \]
    holds for all $f\in L^1(\R^n)$.
\end{theorem}
We would like to mention that, in the proof of Theorem \ref{th: jaming}, the author makes use of the fact that the dual of $L^1(\R^n)$ is $L^\infty(\R^n)$. Subsequently, this problem has been studied in various setting. In \cite{Lavanya_Weyl}, the authors obtained a characterization of the Weyl transform, which they used to prove a characterization theorem for the group Fourier transform on the Heisenberg group in \cite{Lavanya_Heisen}. The characterization of the group Fourier transform on an arbitrary non-abelian compact group is obtained in \cite{Kumar}. Recently, in \cite{krukowski}, 
M. Krukowski has shown that two fundamental properties of the Fourier transform, namely the time differentiation property and the Dirac delta property are sufficient to characterize the Fourier transform among integral transforms.

\par In \cite{Hasan1}, the characterization problem is investigated for the Fourier transform of Hilbert space-valued functions on $\R$. Let $H$ be a separable Hilbert space. Then the convolution of $H$-valued functions is defined in \cite{Hasan1} as follows.
\begin{equation} \label{eq: vector conv def}
    (f\ast g) (x)=\int_{\R}\left\langle f(y), g(x-y)\right\rangle dy,~~f,g\in L^1(\R,H).
\end{equation}
Observe that for $f,g\in L^1(\R,H)$, $f\ast g\in L^1(\R)$. So, the authors in \cite{Hasan1} defined convolution of a scalar valued function and a vector valued function as follows.
\begin{equation}\label{eq: vector scalar conv}
    (f\odot g)(x)=\int_\R f(x-y)g(y)dy,~~f\in L^1(\R), g\in L^1(\R,H).
\end{equation}
Then it is easy to see that $(f\ast g)\odot h\in L^1(\R, H)$ whenever $f,g,h\in L^1(\R,H)$. Further, for $f,g:\R\to H$, $h:\R\to\C$, we define the following notion of products.
\[
(f\times g)(x)=\langle f(x),g(-x)\rangle, \text{ and }\,\, (h\smblksquare g)(x)=h(x)g(x),~~x\in\R.
\]
Then we have the following convolution theorem in this context.
\begin{equation}\label{eq: conv th intro}
    \left(\left( f\ast g\right)\odot h\right)~\widehat{}~= \left(\widehat{f}\times \widehat{g}\right) \smblksquare\widehat{h},
\end{equation}
$f,g,h\in L^1(\R,H),$ where the Fourier transform $\widehat{f}$ of a vector valued function $f$ is defined as in the scalar valued case.
\begin{equation*}
    \widehat{f}(\xi)=\int_\R e^{-2\pi ix\xi}f(x)dx,~~\xi\in R,
\end{equation*}
where the right hand side is a Bochner integral. For a detailed study of vector-valued integrations, we refer to \cite{Diestel}. 

% Let us now state the characterization theorem in the context of vector valued functions. 

% \begin{theorem}[\cite{Hasan1}]\label{L1 char}
% Let $(H, \langle {\cdot, \cdot } \rangle)$ be a separable Hilbert space with an orthonormal basis $\{e_k: k\in\N\}$ and $T:L^1(\R, H)\rightarrow C_0(\R, H)$ be a bounded linear operator satisfying
% \item[(i)] $T((f\ast g)\odot h)=\left(Tf\times Tg\right) \smblksquare Th$, for $f, g, h \in L^1(\R, H)$,
% \item[(ii)]$T(\tau_s f)(\xi)=e^{-2\pi i s \xi}Tf(\xi)$, $f\in L^1(\R, H), ~~s, \xi \in \R$,
% \item[(iii)] If $\langle {f(x), e_k}\rangle =0$ for $a.e. ~x \in \R$ for some $k$ then $\langle {Tf(\xi), e_k}\rangle =0$ for all $ \xi \in \R$, and
% \item[(iv)] $T(\widetilde {f})(-\xi)=Tf(\xi)$,  $f\in L^1(\R, H),~~ \xi \in \R$, where $\widetilde {f}(x)=f(-x)$.\\
% Then there exists $E\subset \R$ such that
% \begin{center}
% $\langle {Tf(\xi), Tg(\xi)}\rangle =\chi_E(\xi)\langle {\widehat{f}(\xi), \widehat{g}(\xi)}\rangle$,\quad  $f,g \in L^1(\R, H), ~ \xi \in \R$.
% \end{center}
% \end{theorem}
% As a corollary to Theorem \ref{L1 char}, the following characterization is obtained.
% \begin{corollary}[\cite{Hasan1}]\label{cor: L^1 cor}
% If in addition to the hypothesis of Theorem \ref{L1 char}, $T$ satisfies $T(g_k)=g_k$ where $g_k(x)=e^{-\pi x^2}e_k$ for all $k\in \N$, then there exists $E\subset \R$ such that $Tf(\xi)=\chi_E (\xi) \widehat{f}(\xi)$,\quad $f \in L^1(\R, H),~~\xi \in \R$.
% \end{corollary}
In this paper, our aim is to obtain a characterization of the Fourier transform on the $H$-valued Schwartz space, $S(\R^n, H)$, where $H$ is a separable Hilbert space. The space $S(\R^n, H)$ is defined as in the classical case as follows. 

\begin{definition}[\cite{VVFT}]
  Let $H$ be a Hilbert space. $S(\mathbb{R}^n, H)$, the $H$-valued Schwartz space is the collection of all rapidly decreasing smooth function $f: \R^n \to H$ such that $(1+|x|)^m\|D^\alpha f(x)\|_H$ is bounded for all multi-indices $\alpha$ and all non negative integer $m$, where
  $$D^\alpha f(x)= \frac{\partial^{|\alpha|}f(x)}{\partial x_1^{\alpha_1} \partial x_2^{\alpha_2}\cdots \partial x_n^{\alpha_n}},$$
  and $ |\alpha|= \alpha_1+ \alpha_2+ \cdots + \alpha_n$ is the order of the multi-index $\alpha = (\alpha_1,\alpha_2, \cdots, \alpha_n).$
\end{definition}
The space $S(\R^n,H)$ turns out to be a topological vector space equipped with the family of semi-norms
\[
p_{m,\alpha}(f)=\sup_{x\in\R^n}(1+|x|)^m\|D^\alpha f(x)\|_H.
\]
Let $S'(\mathbb{\R}^n,H)$ be the topological dual of $S(\mathbb{R}^n,H)$. For convenience, we write $S_H$ and $S_H^\prime$ for $S(\R^n, H)$ and $S^\prime(\R^n, H)$, respectively. When $H=\C$, we write $S_\C$ and $S_\C^\prime$ to denote the classical Schwartz space and its dual, respectively.
\par Proceeding as in the classical case, one can show that the Fourier transform leaves $S_H$ invariant. Moreover, using the inversion formula and the Plancherel theorem (see Proposition 2.4.5 and Theorem 2.4.9 in \cite{VVFT}) for the Fourier transform of vector valued functions it follows that the Fourier transform is a linear isomorphism on $S_H$. Moreover, it is easy to see that $\widehat{\widehat{f}}(\xi)=f(-\xi)$, for all $f\in S_H$. Using this fact and \eqref{eq: conv th intro} it follows that
\begin{equation}\label{eq: intro conv th 2}
\left((f\times g)\smblksquare h\right)~\widehat{}~=\left(\widehat{f}\ast \widehat{g}\right)\odot\widehat{h},
\end{equation}
$f,g,h\in S_H$.
\par The Fourier transform $\widehat{f}$, of $f\in S_H^\prime$ is defined via the following relation.
\[
\langle\widehat{f},\phi\rangle=\langle f,\widehat{\phi}\rangle,~~\text{ for all }\phi\in S_H.
\]
Then it is easy to see that $\widehat{f}\in S_H^\prime$, whenever $f\in S_H^\prime$. Furthermore, it follows that the Fourier transform maps $S_H^\prime$ onto $S_H^\prime$ bijectively and the inverse $\widecheck{f}$ of $f\in S_H^\prime$ is given by the following equality.
\[
\langle \widecheck{f}, \phi\rangle=\langle f, \widecheck{\phi}\rangle,~~\text{ for all }\phi\in S_H,
\]
where $\widecheck{\phi}$ is the inverse Fourier transform of $\phi$.
\par To obtain the desired characterization, we first extend the notion of convolution of $f\in S_H^\prime, g,h\in S_H$ defined in \eqref{eq: vector conv def} and \eqref{eq: vector scalar conv}. More precisely, $(f\fisheye g)\circledast h$ is defined as an element in $S_H^\prime$ as follows.
\begin{equation}\label{eq: intro conv def}
\left \langle(f\fisheye g)\circledast h, \phi\right\rangle=\left\langle f, (\phi\ast \widetilde{h})\mdsmblkcircle \widetilde{g}\right\rangle,~~~\phi\in S_H,
\end{equation}
where $\widetilde{g}(x)=g(-x)$. The easily verifiable fact $(\phi\ast h)\mdsmblkcircle g\in S_H$, whenever $\phi,g,h\in S_H$, implies that the above definition makes sense. Moreover, we extend the notion of product in a similar way as follows.
\begin{equation}\label{eq: intro pro def}
\langle (f\blackinwhitesquare g)\otimes h,\phi\rangle=\langle f, (\phi\times h)\smblksquare g\rangle,~~ f\in S_H^\prime, g,h,\phi\in S_H.
\end{equation}
For $f\in S_H^\prime, g,h\in S_H$, using \eqref{eq: conv th intro} and \eqref{eq: intro conv th 2} we obtain the following.
\begin{align}
   \label{eq: intro conv pro} \left((f\fisheye g)\circledast h\right)~\widehat{}~&=\left(\widehat{f}\,\blackinwhitesquare\widehat{g}\right)\otimes\widehat{h},\\\label{eq: intro pro conv}\left((f\blackinwhitesquare g)\otimes h\right)~\widehat{}~&=\left(\widehat{f}\,\fisheye \widehat{g}\right)\circledast \widehat{h}.
\end{align}
We now look for another interesting property of the Fourier transform in the context of vector valued functions. Let $\{e_k:k\in\N\}$ be an orthonormal basis for $H$. Then the Fourier transform of vector valued functions satisfies the following equality.
$$
\F\left(f \right)=\sum_{k\in\N} \F\left(\langle f(\cdot),e_k\rangle e_k\right),$$ 
for all $f\in S_H$ with pointwise expansion $f(x)=\sum_{k\in\N}\langle f(x), e_k\rangle e_k$. In fact, we have
$$
\F(f)(\xi)=\sum_{k\in\N}\left\langle\F(f)(\xi),e_k\right\rangle e_k=\sum_{k\in\N}\int_{\R^n}e^{-2\pi ix\xi}\langle f(x), e_k\rangle e_k dx.
$$

\par In the first half of this paper, our aim is to obtain a characterization of the Fourier transform on $S_H$ using the structural properties of the Fourier transform. In order to state our results, we need the following notion of orthogonality of an element $\Lambda\in S_H^\prime$ and $e_k\in H$.
\begin{definition}
    A tempered distribution $\Lambda\in S_H^\prime$ is said to be orthogonal to $e_k$ if $\Lambda(\phi e_k)=0,~~\text{ for all }\phi\in S_\C.$
\end{definition}
In the context of Fourier transform of vector valued functions on $\R^n$, we obtain the following characterization.
\begin{theorem}\label{th: main th r}
    Let $H$ be a separable Hilbert space with an orthonormal basis $\{e_k:k\in\N\}$. Suppose that $U: S_H \to S_H$ is a bijective mapping with bijective extension (which we denote by $U$ again) on $S_H^\prime$. Assume that $U$ satisfies the following:
    \begin{itemize}
     \item [(i)]  $U((f\fisheye g)\circledast h)=(Uf \blackinwhitesquare Ug) \otimes Uh)$ $\forall f \in S_H'$ and $\forall g,h \in S_H.$
        \item  [(ii)] $U((f\blackinwhitesquare g)\otimes h)=(Uf\fisheye Ug)\circledast Uh)$ $\forall f \in S_H'$ and $\forall g,h \in S_H.$
        \item [(iii)] If $f$ is orthogonal to $e_k$, then $Uf$ is orthogonal to $e_k$,  for all $f\in S_H^\prime$ and $Ue_k \in H$.
     \item[(iv)] $U\widetilde{f}(x)=\widetilde{Uf}(x),$ where $\widetilde{f}(x)=f(-x)$ for all $f \in S_H.$ 
        %\item [(i)] If $f$ is orthogonal to $e_k$, then $Tf$ is orthogonal to $e_k$, for all $f\in S_H^\prime$.
       % \item[(v)] $U\left(f \right)=\sum_{k\in\N} U\left(\langle f(\cdot),e_k\rangle e_k\right)$, for all $f\in S_H$.
    \end{itemize}
    Then there exists $B \in M_n(\R)$ such that 
    $$\left\langle Uf(\omega), Ug(\omega)\right\rangle = \left\langle \F f(B\omega), \F g(B\omega) \right\rangle\,\,\text{ or }\,\, \left\langle Uf(\omega), Ug(\omega)\right\rangle = \left\langle \F g(B\omega), \F f(B\omega) \right\rangle, \,\,\, f,g \in S_H, \omega\in \R^n.$$
\end{theorem} 
%{\color{teal} 
An interesting consequence of our method is that, if the operator $U$ in Theorem \ref{th: main th r} satisfies an additional assumption: $$U(g_k^\varepsilon)(\xi)=g_k(\varepsilon\xi),$$ for all $k\in \N$ and $\varepsilon>0$, where $g_k(x)=e^{-\pi |x|^2}e_k$ , and $g_k^\varepsilon(x)=\varepsilon^{-n}g_k(x/\varepsilon)$ then, the operator $U$ uniquely determines the Fourier transform (see Corollary \ref{cor: main th cor}). 
At this point, we would like to mention that a priori we do not assume anything about the boundedness or linearity of the operator $U$. However, as we are working with the infinite-dimensional Hilbert spaces, we require conditions (iii) and (iv) in our case. 

\begin{remark}
    The use of an orthonormal basis in the hypotheses does not affect the validity of the above characterization theorem, since it serves only to formulate the orthogonality condition. In contrast, defining convolution component-wise with respect to the basis may introduce basis dependence and convergence difficulties, potentially restricting the generality of the result. Hence, we choose the standard convolution of vector-valued functions, which requires a suitable multiplication in the Hilbert space.
\end{remark}

\par Our approach is inspired by the work of Alesker et al. \cite{Alesker_paper}.  Observe that the operator $U$, interchanges convolution and product of functions. However, the operator $T$ defined as the composition of the operator $U$ with the Fourier transform maps convolution (resp. product) of functions into their convolution (resp. product).
%{\color{blue} Observe that the operator $T$ obtained by composing the operator $U$ with a Fourier transform satisfies the hypotheses of Theorem \ref{th: main th r}. 
In the following theorem, we obtain a characterization of $T$, which in turn leads to Theorem \ref{th: main th r}.
%We obtain Theorem \ref{th: main th r} as a natural consequence of the result stated below.
%At first, we obtain the following characterization.
\begin{theorem}\label{th: main th r 1}
    Let $H$ be a separable Hilbert space with an orthonormal basis $\{e_k:k\in\N\}$. Suppose that $T: S_H \to S_H$ is a bijective mapping with a bijective extension (which we denote by $T$ again) on $S_H' $. Assume that $T$ satisfies the following:
    \begin{itemize}
     \item [(i)]  $T((f\fisheye g)\circledast h)=(T f \fisheye T g) \circledast Th)$ $\forall f \in S_H'$ and $\forall g,h \in S_H$.
        \item  [(ii)] $T((f\blackinwhitesquare g)\otimes h)=(T f\blackinwhitesquare Tg)\otimes Th)$ $\forall f \in S_H'$ and $\forall g,h \in S_H$.
        \item [(iii)] If $f\in S_H'$ is orthogonal to $e_k$, then $Tf$ is orthogonal to $e_k$ and $Te_k \in H$.
     \item[(iv)] $T\widetilde{f}(x)=\widetilde{Tf}(x)$, for all $f \in S_H$, where $\widetilde{f}(x)=f(-x)$.
    \end{itemize}
    Then there exists $B \in M_n(\R)$ such that 
    $$\left\langle Tf(\omega), Tg(\omega)\right\rangle = \left\langle f(B\omega), g(B\omega) \right\rangle\,\,\text{ or }\,\, \left\langle Tf(\omega), Tg(\omega)\right\rangle = \left\langle g(B\omega), f(B\omega) \right\rangle,\,\,\, f,g \in S_H,~\omega\in \R^n.$$
\end{theorem}
Since we are working with infinite dimensional Hilbert spaces, we require the additional hypothesis $(iii)$ and $(iv)$ in Theorem \ref{th: main th r 1}. 

In the proof of our main Theorem, the following two types of distributions play an important role. For $y\in \R^n, k\in \N $ we define $\delta_y^k$, as 
\[
\langle \delta_y^k,f\rangle=\langle f(y), e_k\rangle,~~~f\in S_H,
\]
where $\{e_k:k\in \N\}$ is an orthonormal basis for $H$. Moreover, any $u\in H$ can be considered as an element in $S_H^\prime$ via
\[\langle u, f\rangle=\int_{\R^n}\langle f(x), u\rangle dx, ~~~f\in S_H.\]
More precisely, we need to know the action of $T$ on $\delta_y^k,~e_k\in S_H^\prime$. In order to obtain these actions, first we need to extend \eqref{eq: intro conv def} and \eqref{eq: intro conv pro} for above mentioned distributions. Further, we need to derive the action of $T$ on such extended convolution and product (similar to the hypothesis (i), (ii) in Theorem \ref{th: main th r}). In the following section, we define and prove some properties of convolution of three distributions.

%{\color{blue}
The motivation for studying vector-valued analysis is to extend classical analytical methods to functions taking values in infinite-dimensional Hilbert spaces. Such extensions provide a natural framework for investigating structural properties of bijective mappings that preserve operations such as convolution, tensor products, orthogonality, and reflection symmetry. The characterization of these mappings in terms of suitable linear transformations or canonical operators forms a natural part of the broader study of symmetry and structure in functional analysis and harmonic analysis. These questions also highlight the interplay between operator theory, distribution theory, and harmonic analysis, see \cite{benedek1962convolution, RubioDeFranciaRuizTorrea1986}. Recently, in \cite{BiswasRadha2024}, the vector-valued Fourier transform was used to convert translations into multiplication by scalar exponentials while preserving the inner product. Together with Parseval's identity, this yields frame characterizations in terms of the Fourier transforms of the generating functions.
%The motivation for studying vector-valued analysis is to extend classical analysis to functions taking values in an infinite-dimensional Hilbert space. The study of bijective mappings that preserve operations such as convolution, tensor products, orthogonality, and reflection symmetry helps establish representation theorems that characterize these mappings in terms of suitable linear transformations or canonical operators. Such investigations strengthen the connections between functional analysis, distribution theory, and harmonic analysis.
%\par
%The use of an orthonormal basis in the hypotheses does not affect the validity of the above characterization theorem, since it serves only to formulate the orthogonality condition. In contrast, defining convolution component-wise with respect to the basis may introduce basis dependence and convergence difficulties, potentially restricting the generality of the result. Hence, we choose the standard convolution of vector-valued functions, which requires a suitable multiplication in the Hilbert space.
\par
Our result provides a more general characterization in a concise form. Indeed, although a separable Hilbert space admits an orthonormal basis, reducing the problem to the scalar-valued case via a coordinate-wise expansion does not, by itself, establish the result in full generality. 
%Such an approach requires justifying the convergence of the resulting infinite series in the \(H\)-norm and the interchange of the relevant operators with the series. 
Moreover, estimates with scalar values alone may not fully capture the underlying structure with a Hilbert-space value.

%{\color{blue}
We structure the proof in several steps, following the approach in \cite{Alesker_paper}. Unlike the classical setting, one of the key assumptions in our framework involves two different convolution operations.
%We prove the theorem in several steps as in \cite{Alesker_paper}. In contrary to the classical case, one of the hypotheses in our case consists of two different convolutions.
As a result, our proof differs significantly from the classical case and presents additional challenges. 
%As a result, our proof became different from that of the classical case.
In some special cases, we compute pointwise value of a distribution by applying it on a Schwartz class function. Furthermore, we make essential use of the fact that two distributions $\Lambda$ and $\Lambda'$ are equal if and only if $(\Lambda\fisheye\phi_1)\circledast\phi_2=(\Lambda'\fisheye\phi_1)\circledast\phi_2$, for all $\phi_1,\phi_2\in S_H$. This fact follows from the definition of the convolutions.
% Proofs consits of sevarel steps. Can you write the difficulties and nobility of the proof of the main theorem.
% Since we are dealing with infinite-dimensional Hilbert spaces, each step require a new idea.
%}
In the Section 3, we present the proof of our main theorem. 
\par 
%{\color{blue} 
Another objective of this paper is to address the problem in discrete setting, specifically the case where the underlying group is the finite cyclic group $\Z_n$. In \cite{Jaming}, the Fourier transform of functions on $\Z_n$ was characterized, which is analogous to the results of Alesker et al. \cite{Alesker_paper} in the continuous setting. In the continuous setting, Alesker et al. have assumed that  the operator $T$ maps smooth functions to smooth functions. However, since the notion of smoothness does not carry over to the finite group $\Z_n$, Jaming (\cite{Jaming}) instead assumed a mild continuity condition and obtained the characterization of Fourier transform on $L^1(\Z/n\Z)$. 

Note that the Fourier transform of $a\in L^1(\Z_n)$ is given by
\[
\F(a)(\ell)=\sum_{j=0}^{n-1} a(j)e^{-\frac{2\pi ij\ell}{n}},\,\, \ell\in\Z_n.
\]
A straightforward computation shows that $\F(f\cdot g)=\frac{1}{n}\F(f)\ast \F(g)$
holds for all $f,g\in L^1(\Z_n)$, where $$(f\ast g)(\ell)=\displaystyle\sum_{j=0}^{n-1} f(j) g(\ell-j) \text{ and } (f\cdot g)(\ell)=f(\ell) \cdot g(\ell),~~\ell\in\Z_n.$$
Due to the presence of the additional factor $\frac{1}{n}$ in the product-to-convolution formula above, we present the result in Jaming (\cite[Corollary 2.3]{Jaming}) with the necessary changes. Here we endow the physical space $\Z/n\Z$ and the frequency space $\Z/n\Z$ with the counting measure.
\begin{theorem}[\cite{Jaming}]
    Let $T: L^1(\Z/n\Z) \to L^1(\Z/n\Z)$ be a bijective transformation (not necessarily linear) such that the map $\C \to L^1(\Z/n\Z)$, $c\to T(c\mathbb{1})$ is continuous. Assume that
    \begin{itemize}
        \item [(i)] for every $a,b \in L^1(\Z/n\Z)$, $T(a \cdot b)= \frac 1nT(a)\ast T(b)$
        \item [(ii)] for every $a,b \in L^1(\Z/n\Z)$, $T(a\ast b)=T(a) \cdot T(b)$.
    \end{itemize}
    Then there exists $\eta \in \{1,2, \cdots n-1\}$ that has no common divisor with $n$ such that
    \begin{itemize}
        \item [-] either, for each $j \in \Z/n\Z$, for each $a \in L^1(\Z/n\Z)$, $T(a)(\eta j)= \F a(j)$,
        \item [-] or, for each $j \in \Z/n\Z$, for each $a\in L^1(\Z/n\Z)$, $ T(a)(\eta j)= \overline{\F a(j)}$.
    \end{itemize}
    \end{theorem}
    
%Because of the extra factor $\frac 1n$ present in the above product to convolution formula, we found a typo in the result of Jaming (\cite[Corollary 2.3.]{Jaming}). Here, we state his result with this modification:

A natural question then arises: can these results be extended in the context of Hilbert space valued functions? In this paper, we provide a characterization of Fourier transform on
$L^1(\Z/n\Z, H)$, where $H$ is a  separable Hilbert space. As in \cite{Jaming}, we impose a mild continuity condition to obtain the characterization.
Furthermore, by applying a tensorization argument, our result naturally extends to functions defined on arbitrary finite abelian groups.
%}

In the second half of our paper, we characterize the Fourier transform of vector valued functions on the group $\Z_n=\Z/n\Z$. Towards this end, first we define the convolution of such functions and obtain convolution theorem. Finally, we prove a characterization theorem for the Fourier transform 
%{\color{red}\st{of vector valued functions}}
on $L^1(\Z/n\Z, H)$. 
%{\color{red}\st{In this case, we obtain the following.}} 
In order to state the main result in this setting, we need the following notion of orthogonality. An element $a=(a_1,\cdots,a_{n-1})\in L^1(\Z_n,H)$ is said to be orthogonal to $u\in H$ if $\langle a_l,u\rangle=0,\,l=0,\cdots, n-1$. We use $a\perp u$ to say that $a$ is orthogonal to $u$.
\begin{theorem}
    Let $U: L^1(\Z_n, H) \to L^1(\Z_n,H)$ be a bijective transformation (not necessarily linear) such that the map $c\mapsto T(\mathbb{1}_{ce_k})$ is continuous, for any $c\in\C$ and $T(\mathbb{1}_k)$ is a constant. Assume that
    for all $a,b,c \in L^1(\Z_n,H)$ 
    \begin{itemize}
        \item [(i)] $U((a\times b)\smblksquare c)= \frac{1}{n^2}(U(a)\ast U(b))\mdsmblkcircle U(c).$
        \smallskip
        \item [(ii)] $U((a\ast b)\mdsmblkcircle c)=(U(a)\times U(b))\smblksquare U(c).$
        \smallskip
        
 		%\item[(iii)] $U(a)=\sum_{k\in\N}U(\langle a,e_k\rangle e_k)$, where $a$ has the pointwise orthonormal expansion $a_j=\sum_{k\in\N}\langle a_j,e_k\rangle e_k,~j\in\Z_n$.
        
        \smallskip
        
        \item[(iii)] If $a\perp e_k$ for some $k\in\N$, then $T(a)\perp e_k$.
        
        \smallskip
        
       % \item[(iv)] The map $c\mapsto T(\mathbb{1}_{ce_k})$ is continuous, for any $c\in\C$ and $T(\mathbb{1}_k)$ is a constant.
        
        \smallskip
        
 		\item[(iv)] $U(b_k)=b_k$, for all $k\in\N$, where $a_k(j)= e_k$ for all $j=0,1,\cdots,n-1$ and $b_k=a_k+\F a_k$. 
\end{itemize} 		
 		Then there exists $\eta \in \{1,2, \cdots, n-1\}$ that has no common divisor with $n$ such that 
    \begin{itemize}
        \item [-] either, $U a(\eta j)= \F a(j)$ for all $a \in L^1(\Z_n,H)$ and $j\in \{0,1,\cdots, n-1\}.$ 
        
        \smallskip
        
        \item [-] or, $U a(\eta j)= \displaystyle\sum_k \overline{\langle \F ( a)(j), e_k\rangle}e_k$ for all $a\in L^1(\Z_n,H)$ and $j\in \{0,1,\cdots ,n-1\}$.
    \end{itemize}  
    \end{theorem}

\section{Convolution of Dirac delta distributions}
%Recall that the convolution of vector valued functions was defined in \cite{Hasan1}. The convolution $(f\ast g)\fisheye h$ is a map from $S_H\times S_H\times S_H$ into $S_H$. We have extended the notion of convolution as a map from $S_H'\times S_H\times S_H$ into $S_H'$. 
Recall that we have defined the convolution $(f\fisheye g)\circledast h$ when $f\in S_H^\prime,g,h\in S_H$. However, convolution with Dirac delta distributions exhibits some interesting properties. In this section, we establish several useful identities that involve such convolutions. These identities play a key role in the proof of our main result, which is also of independent interest and complements the results of \cite{Alesker_paper,Alesker_BC}.
%However, something interesting happens when we take convolution of functions with Dirac delta distributions. In this section, we derive some useful identities involving convolution of Dirac delta distributions. 
We define the derivative of $\Lambda\in S_H'$ as follows.\\
$$\langle D^\alpha\Lambda,\phi\rangle=(-1)^{|\alpha|}\langle\Lambda, D^\alpha\phi\rangle,~~\phi\in S_H.$$ We now show that for any $x,y\in\R^n$, we can define $(\delta_x^k\circledast \delta_y^k)\fisheye \delta_0^k$ appropriately. \\
For $\phi,\,f\in S_H$, recall that $\left \langle (\delta_y^k \fisheye \phi)\circledast f, \psi\right \rangle=\left\langle \delta_y^k, (\psi \ast \widetilde{f})\mdsmblkcircle \widetilde{\phi}\right\rangle$, $\psi\in S_H$. Then, an easy computation shows that  
\begin{equation}\label{eq: delta conv 1}
(\delta_y^k \fisheye \phi)\circledast f(t)=\int_{\R^n} \left\langle e_k, \phi(z-y)\right\rangle f(t-z)\, dz=\int_{\R^n}\langle e_k, \phi(z)\rangle f(t-z-y)dz,
\end{equation}
%By making a change of variables, we can write $(\delta_y^k \fisheye \phi)\circledast f(t)=\phi_k\mdsmblkcircle f(t-y)$, 
from which it follows that $(\delta_y^k \fisheye \phi)\circledast f\in S_H$.
This suggests that, for $\phi\in S_H$, we can define $(\delta_0^k \circledast \phi)\fisheye \delta_y^k$ as follows. 
\begin{align*}
    \left\langle (\delta_0^k \circledast \phi)\fisheye \delta_y^k, f \right\rangle &= \left\langle \delta_0^k, (\delta_y^k \fisheye \widetilde{\phi})\circledast f \right\rangle =\left\langle  (\delta_y^k \fisheye \widetilde{\phi})\circledast f (0), e_k \right\rangle = \int_{\R^n}\left\langle e_k, \phi(y+z)\right\rangle \left\langle f(z), e_k \right\rangle \, dz,
\end{align*}
$f\in S_H$. Further, from the above equality, it is clear that $(\delta_0^k \circledast \phi)\fisheye \delta_y^k\in S_H$. So, we define $((\delta_x^k \circledast \delta_y^k)\fisheye \delta_0^k)\in S_H^\prime$ as follows.
\begin{align*}
    \left\langle (\delta_x^k \circledast \delta_y^k)\fisheye \delta_0^k, \phi \right\rangle
    = \left\langle \delta_x^k , (\delta_0^k \circledast \phi) \fisheye \delta_y^k \right\rangle,~~\text{ for $\phi\in S_H$.}
    %= \left\langle (\delta_0^k \circledast \phi) \fisheye \delta_y^k (x), e_k \right\rangle= \left\langle \phi(x+y), e_k \right\rangle= \left\langle \delta_{x+y}^k , \phi \right\rangle.
\end{align*}
More specifically, we have the following in the context of the convolution of Dirac delta distributions. 
\begin{propn}\label{pro: delta derivative}
For multi-indices $\alpha,\beta$, we have  $(D^\alpha \delta_x^k\circledast D^\beta \delta_y^k)\fisheye \delta_0^k= D^{\alpha+\beta}\delta_{x+y}^k$. In particular,
$( \delta_x^k\circledast  \delta_y^k)\fisheye \delta_0^k= \delta_{x+y}^k.$
\end{propn}
\begin{proof}
%We define the distribution $(D^\alpha \delta_x^k\circledast D^\beta \delta_y^k)\fisheye \delta_0^k$ as  $\left\langle (D^\alpha \delta_x^k\circledast D^\beta \delta_y^k)\fisheye \delta_0^k, \phi \right\rangle = \left\langle D^\alpha \delta_x^k, (\delta_0^k \circledast \phi)\fisheye D^\beta \delta_y^k \right\rangle $, $\phi\in S_H$. In order to show that this definition makes sense we need to verify that $(\delta_0^k \circledast \phi)\fisheye D^\beta \delta_y^k \in S_H.$
For $f,\,\phi,\,\psi\in S_H$, consider
\begin{align*}
    \left\langle \left(D^\beta \delta_y^k \fisheye \widetilde{\phi}\right)\circledast \psi, f \right\rangle &= \left\langle D^\beta \delta_y^k, \left(f\ast \widetilde{\psi}\right)\mdsmblkcircle \phi \right\rangle\\
    &= (-1)^{|\beta|} \left\langle \delta_y^k, D^\beta \left(\left(f\ast \widetilde{\psi}\right)\mdsmblkcircle \phi\right)\right\rangle\\
    &= (-1)^{|\beta|} \left\langle \delta_y^k, \left(f\ast \widetilde{\psi}\right)\mdsmblkcircle D^\beta \phi \right\rangle\\
    &= (-1)^{|\beta|} \left\langle \left(f\ast \widetilde{\psi}\right)\mdsmblkcircle D^\beta \phi(y), e_k \right\rangle\\
    %&= (-1)^{|\beta|} \left\langle \int_{\R^n}  \int_{\R^n}  \left\langle f(u), \widetilde{\psi}(z-u)\right\rangle D^\beta \phi(y-z)\, du \, dz, e_k \right\rangle\\
    &= \int_{\R^n}  \left\langle f(u), \int_{\R^n}  (-1)^\beta \langle e_k, D^\beta \phi(y-z) \rangle \psi(u-z)\, dz \right\rangle \, du.
\end{align*}
Hence, $(D^\beta \delta_y^k \fisheye \widetilde{\phi})\circledast \psi(u)= (-1)^{|\beta|} \int_{\R^n}  \langle e_k, D^\beta \phi(y-z)\rangle \psi(u-z)\, dz$ holds pointwise. Writing the right hand side as a convolution of $D^\beta\phi$ and $\psi$ it follows that $\left(D^\beta \delta_y^k \fisheye \widetilde{\phi}\right)\circledast \psi\in S_H$. This suggests that we define $\left(\delta_0^k \circledast \phi\right)\fisheye D^\beta \delta_y^k$ for $\phi\in S_H$ as follows.
$$
\left\langle \left(\delta_0^k \circledast \phi\right)\fisheye D^\beta \delta_y^k, \psi \right\rangle = \left\langle \delta_0^k, \left(D^\beta \delta_y^k \fisheye \widetilde{\phi}\right)\circledast \psi \right\rangle,~~\, \text{ for $\psi\in S_H$.}
$$
%Now we can calculate the following. 
% \begin{align*}
%     \left\langle \left(\delta_0^k \circledast \phi\right)\fisheye D^\beta \delta_y^k, \psi \right\rangle &= \left\langle \delta_0^k, \left(D^\beta \delta_y^k \fisheye \widetilde{\phi}\right)\circledast \psi \right\rangle\\
%     &= \left\langle \left(D^\beta \delta_y^k \fisheye \widetilde{\phi}\right)\circledast \psi(0), e_k \right\rangle\\
%     &= \left\langle (-1)^\beta \int \langle e_k, D^\beta \phi(y-z)\rangle \psi(-z)\, dz, e_k \right\rangle\\
%     &= (-1)^\beta \int \left\langle e_k, D^\beta \phi(y-z)\right\rangle \left\langle \psi(-z), e_k \right\rangle \, dz\\
%     &= \int \left\langle \psi(z), (-1)^\beta \left\langle D^\beta \phi (y+z), e_k \right\rangle e_k \right\rangle \, dz
% \end{align*}
Then a simple calculation shows that we can write $(\delta_0^k \circledast \phi)\fisheye D^\beta \delta_y^k(z)= (-1)^{|\beta|} \langle D^\beta \phi (y+z), e_k \rangle e_k$ pointwise, from which it is clear that $(\delta_0^k \circledast \phi)\fisheye D^\beta \delta_y^k\in S_H$. Thus we define $\left(D^\alpha \delta_x^k\circledast D^\beta \delta_y^k\right)\fisheye \delta_0^k$ as follows. 
$$
\left\langle \left(D^\alpha \delta_x^k\circledast D^\beta \delta_y^k\right)\fisheye \delta_0^k, \phi \right\rangle = \left\langle D^\alpha \delta_x^k, (\delta_0^k \circledast \phi)\fisheye D^\beta \delta_y^k \right\rangle,~~\text{ for $\phi\in S_H$}.
$$
Then
\begin{align*}
    \left\langle \left(D^\alpha \delta_x^k\circledast D^\beta \delta_y^k\right)\fisheye \delta_0^k, \phi \right\rangle &= \left\langle D^\alpha \delta_x^k, (\delta_0^k \circledast \phi)\fisheye D^\beta \delta_y^k \right\rangle\\
    &= (-1)^{|\alpha|} \left\langle \delta_x^k, D^\alpha \left(\left(\delta_0^k\circledast \phi\right)\fisheye D^\beta \delta_y^k\right)\right\rangle\\
    &= (-1)^{|\alpha|}\left\langle D^\alpha \left(\left(\delta_0^k\circledast \phi\right)\fisheye D^\beta \delta_y^k\right)(x), e_k \right\rangle\\
    %&= (-1)^{|\alpha|}\left\langle D^\alpha \left((-1)^{|\beta|}\left\langle D^\beta \phi(y+.), e_k \right\rangle e_k\right)(x), e_k \right\rangle\\
    &= (-1)^{|\alpha|}\left\langle (-1)^{|\beta|}\left\langle D^{\alpha+\beta}\phi(y+x), e_k \right\rangle e_k, e_k \right\rangle\\
    %&=(-1)^{|\alpha|+|\beta|}\langle D^{\alpha+\beta}\phi(y+x), e_k \rangle\\
     %&= (-1)^{|\alpha+\beta|}\langle \delta_{x+y}^k, D^{\alpha+\beta}\phi \rangle\\
     &=\langle D^{\alpha+\beta}\delta_{x+y}^k, \phi\rangle,
\end{align*}
using $D^\alpha (g \fisheye \Delta)= D^\alpha g \fisheye \Delta,\, (g \fisheye D^\beta \delta_y^k)= (-1)^{|\beta|}(D^\beta g \fisheye \delta_y^k)$ and $D^\alpha (\delta_0^k \circledast \phi)= \delta_0^k\circledast D^\alpha \phi$, which are easy to verify.
\end{proof}

We define $(\delta_y^k\circledast \delta_0^k)\fisheye \psi_1 $ as follows.
$$
\left\langle (\delta_y^k\circledast \delta_0^k)\fisheye \psi_1, \phi \right\rangle:= \left\langle \delta_0^k,  (\delta_y^k\fisheye\widetilde{\psi_1})\circledast \phi \right\rangle,~~\phi\in S_H.
$$
Then a similar computations as above shows that $(\delta_y^k \circledast \delta_0^k)\fisheye \psi_1(x)= \left\langle \psi_1(x+y), e_k \right\rangle e_k$, from which it follows that  $(\delta_y^k \circledast \delta_0^k)\fisheye \psi_1\in S_H.$ We now show that following equality holds.
\begin{equation}\label{eq: conv asso 1}
(((\delta_x^k\circledast \delta_y^k)\fisheye \delta_0^k)\fisheye \psi_1)\circledast \psi_2=\left(\delta_x^k\fisheye \left(\left(\delta_y^k\circledast \delta_0^k\right)\fisheye \widetilde{\psi_1}\right)\widetilde{}\,\right)\circledast \psi_2.
\end{equation}
In order to see this, consider
\begin{align*}
    \left  \langle \left(\left(\left(\delta_x^k\circledast \delta_y^k\right)\fisheye \delta_0^k\right)\fisheye \psi_1\right)\circledast \psi_2, \phi \right \rangle &= \left  \langle \left(\delta_{x+y}^k\fisheye \psi_1\right)\circledast \psi_2, \phi \right\rangle \\
    &= \left  \langle \delta_{x+y}^k, \left(\phi\ast \widetilde{\psi_2}\right)\mdsmblkcircle \widetilde{\psi_1}\right\rangle\\
    &= \left  \langle \left(\phi\ast \widetilde{\psi_2}\right)\mdsmblkcircle \widetilde{\psi_1}(x+y), e_k \right\rangle\\
    &= \left  \langle \int_{\R^n} (\phi\ast \widetilde{\psi_2})(z)\widetilde{\psi_1}(x+y-z)\, dz, e_k \right\rangle\\
    %&= \left\langle \int_{\R^n}  \int_{\R^n}  \left\langle \phi(u), \widetilde{\psi_2}(z-u)\right\rangle \widetilde{\psi_1}(x+y-z)\, du \, dz, e_k \right\rangle\\
    &=  \int_{\R^n} \int_{\R^n} \left  \langle \phi(u), \widetilde{\psi_2}(z-u)\right\rangle \left  \langle  \widetilde{\psi_1}(x+y-z), e_k \right\rangle \, du \, dz.
\end{align*}
On the other hand,
\begin{align*}
&\left  \langle\left(\delta_x^k\fisheye \left(\left(\delta_y^k\circledast \delta_0^k\right)\fisheye \widetilde{\psi_1}\right)\widetilde{}\,\right)\circledast \psi_2, \phi \right\rangle \\&= \left  \langle \delta_x^k, \left(\phi \ast \widetilde{\psi_2}\right)\mdsmblkcircle \left(\left(\delta_y^k\circledast \delta_0^k\right)\fisheye \widetilde{\psi_1}\right)\right\rangle\\
    &= \left  \langle  \left(\phi \ast \widetilde{\psi_2}\right)\mdsmblkcircle \left(\left(\delta_y^k\circledast \delta_0^k\right)\fisheye \widetilde{\psi_1}\right)(x), e_k \right\rangle\\
    &= \left  \langle \int_{\R^n}   \left(\phi \ast \widetilde{\psi_2}\right)(z) \left(\left(\delta_y^k\circledast \delta_0^k\right)\fisheye \psi_1\right)(x-z)\, dz, e_k \right\rangle\\
    &= \left  \langle \int_{\R^n}  \int_{\R^n}  \left  \langle \phi(u), \widetilde{\psi_2}(z-u)\right\rangle \left(\left(\delta_y^k\circledast \delta_0^k\right)\fisheye \widetilde{\psi_1}\right)(x-z)\, du \, dz, e_k \right\rangle\\
    &= \int_{\R^n}  \int_{\R^n}  \left  \langle \phi(u), \widetilde{\psi_2}(z-u)\right\rangle \left  \langle \left(\left(\delta_y^k\circledast \delta_0^k\right)\fisheye \widetilde{\psi_1}\right)(x-z), e_k \right\rangle \, du \, dz\\
    &=  \int_{\R^n}  \int_{\R^n}  \left  \langle \phi(u), \widetilde{\psi_2}(z-u)\right\rangle \left  \langle \left  \langle \widetilde{\psi_1}(y+x-z), e_k \right\rangle e_k, e_k \right\rangle \, du \, dz\\
    &= \int_{\R^n}  \int_{\R^n}  \left  \langle \phi(u), \widetilde{\psi}_2(z-u)\right\rangle \left  \langle \widetilde{\psi_1}(x+y-z), e_k \right\rangle \, du \, dz,
\end{align*}
which establishes \eqref{eq: conv asso 1}.
%Hence, pointwise we can write $ (\delta_0^k \circledast \phi)\fisheye \delta_y^k(z)= \left\langle \phi(y+z), e_k \right\rangle e_k$.
% Observe that
% \begin{align*}
%     (\delta_0^k \fisheye \psi_1)\circledast f_1(y) = \int \left\langle e_k, \psi_1(x)\right\rangle f_1(y-x)\, dx= \int g(x)f_1(y-x)\, dx= g\mdsmblkcircle f_1(y), 
% \end{align*}
% Where $g(x)= \left\langle e_k, \psi_1 \right\rangle$. Hence, $(\delta_0^k \fisheye \psi_1)\circledast f_1 \in S_H $
Similarly, we can show that 
\begin{equation}\label{eq: conv asso 2}
\left(\delta_y^k \fisheye \left(\left(\delta_0^k\fisheye \psi_1\right)\circledast f_1\right)\right)\circledast f_2=\left(\left(\left(\delta_y^k\circledast \delta_0^k\right)\fisheye \widetilde{\psi_1}\right)\widetilde{ }\,\ast f_1\right)\mdsmblkcircle f_2.
\end{equation}

\par We now prove some associativity, which we shall use in proving the main result.

\begin{lemma}
    For $\psi_1, f_1, f_2\in S_H$, we have the following equalities.\\ 
\begin{align}
(i) \label{eq: conv asso 3}\left(D^\alpha \delta_x^k\fisheye \left(\left(\delta_0^k\fisheye \psi_1\right)\circledast f_1\right)\right)\circledast f_2&= \left(\left(\left(D^\alpha \delta_x^k\circledast \delta_0^k\right)\fisheye \widetilde {\psi_1}\right)\widetilde{}~\,\ast f_1\right)\mdsmblkcircle f_2.\\
(ii)\label{eq: conv asso 4}
\left(D^\alpha\delta_x^k\fisheye \left(\left(D^\beta \delta_y^k\circledast \delta_0^k\right)\fisheye \widetilde {f_1}\right)\widetilde{}~\, \right) \circledast f_2&= \left(\left(\left(D^\alpha \delta_x^k \circledast D^\beta \delta_y^k\right)\fisheye \delta_0^k\right)\fisheye f_1\right)\circledast f_2. 
\end{align}
\end{lemma}
\begin{proof}
(i) Observe that, for $\phi\in S_H$, $$(\delta_x^k\fisheye D^\alpha\psi_1)\circledast \phi(t)=\int_{\R^n} \left\langle e_k, D^\alpha \psi_1(s-x)\right\rangle \phi(t-s)\, ds.$$ Then,
\begin{align*}
    \left\langle \left(D^\alpha\delta_x^k\circledast \delta_0^k\right)\fisheye \widetilde {\psi_1}, \phi\right\rangle &= \left\langle\delta_0^k, \left( D^\alpha\delta_x^k\fisheye\widetilde {\widetilde {\psi_1}}\right)\circledast\phi\right\rangle\\
    &=\left\langle\left(D^\alpha\delta_x^k\fisheye\psi_1\right)\circledast\phi(0),e_k\right\rangle\\
    &=\left\langle\left(\delta_x^k\fisheye D^\alpha\psi_1\right)\circledast\phi(0),e_k\right\rangle\\
    &= \left\langle\int_{\R^n} \left\langle e_k, D^\alpha\psi_1(s-x)\right\rangle\phi(-s)\, ds, e_k\right\rangle\\
    &=\int_{\R^n} \left\langle e_k, D^\alpha\psi_1(s-x)\right\rangle\left\langle\phi(-s),e_k\right\rangle\, ds\\
    &= \int_{\R^n} \left\langle e_k, \widetilde{D^\alpha\psi_1}(s+x)\right\rangle\left\langle\phi(s),e_k\right\rangle\, ds\\
    &= \int_{\R^n} \left\langle\phi(s),\left\langle \widetilde{D^\alpha\psi_1}(s+x),e_k \right\rangle e_k\right\rangle\, ds.
\end{align*}
So, pointwise we can write $\left(D^\alpha\delta_x^k\circledast \delta_0^k\right)\fisheye \widetilde {\psi_1}(s)= \left\langle \widetilde{D^\alpha\psi_1}(s+x),e_k \right\rangle e_k= (-1)^{|\alpha|}\left\langle D^\alpha \widetilde {\psi_1}(s+x),e_k \right\rangle e_k.$

% $$D^\alpha ((\delta_x^k \circledast \delta_0^k)\fisheye \widetilde {\psi_1})= \langle D^\alpha \widetilde {\psi_1}(x+.), e_k \rangle e_k.$$
Now, using $D^\alpha((\Delta\fisheye f)\circledast g)= (D^\alpha \Delta \fisheye f)\circledast g$ for all $\Delta\in S_H'$, $f,g \in S_H$ and \eqref{eq: conv asso 2} it follows that
\begin{align*}
    \left(D^\alpha \delta_x^k \fisheye \left(\left( \delta_0^k \fisheye \psi_1\right)\circledast f_1\right)\right)\circledast f_2 &= D^\alpha \left( \left( \delta_x^k \fisheye \left(\left( \delta_0^k\fisheye \psi_1\right)\circledast f_1\right)\right)\circledast f_2\right) \\
    &= D^\alpha \left(\left(\left(\left(\delta_x^k\circledast \delta_0^k\right)\fisheye\widetilde {\psi_1}\right)\widetilde{}~\,\ast f_1\right)\mdsmblkcircle f_2\right)\\
    &= \left( D^\alpha \left(\left(\left(\delta_x^k\circledast \delta_0^k\right)\fisheye\widetilde {\psi_1}\right)\widetilde{}~\,\right)\ast f_1\right)\mdsmblkcircle f_2\\
    &= (-1)^{|\alpha|} \left(\left( D^\alpha \left(\left(\delta_x^k\circledast \delta_0^k\right)\fisheye\widetilde {\psi_1}\right)\right)\widetilde{}~\,\ast f_1\right)\mdsmblkcircle f_2\\
    &= (-1)^{|\alpha|} \left(\left(\left\langle D^\alpha \widetilde {\psi_1}(x+\cdot), e_k\right\rangle e_k\right)\widetilde{}~\,\ast f_1\right)\mdsmblkcircle f_2\\
    &=  \left(\left((-1)^{|\alpha|}\left\langle D^\alpha \widetilde {\psi_1}(x+\cdot), e_k\right\rangle e_k\right)\widetilde{}~\,\ast f_1\right)\mdsmblkcircle f_2\\
    &= \left(\left(\left( D^\alpha \delta_x^k\circledast \delta_0^k\right)\fisheye \widetilde {\psi_1}\right)\widetilde{}~\,\ast f_1\right)\mdsmblkcircle f_2.
\end{align*}

%Using the definition of derivative and the continuity of the inner product, we can calculate that $D^\alpha ((\delta_x^k \circledast \delta_0^k)\fisheye \widetilde {\psi_1})= \langle D^\alpha \widetilde {\psi_1}(x+.), e_k \rangle e_k$. We have used it on the above. 

(ii) Appealing to Proposition \ref{pro: delta derivative}, we have
$$D^{\alpha+\beta}\left(\left(\delta_x^k\circledast\delta_y^k\right)\fisheye\delta_0^k\right)= \left(D^\alpha\delta_x^k\circledast D^\beta\delta_y^k\right)\fisheye \delta_0^k.$$
Using this and an easy calculations yields,
\begin{align*}
    \left( D^\alpha \delta_x^k \fisheye \left(\left( D^\beta\delta_y^k \circledast \delta_0^k\right)\fisheye \widetilde {f_1}\right)\widetilde{}~\, \right)\circledast f_2 &= D^\alpha \left(\delta_x^k \fisheye \left(\left( D^\beta\delta_y^k \circledast \delta_0^k\right)\fisheye \widetilde {f_1}\right)\widetilde{}~\, \right)\circledast f_2\\
    &= D^\alpha \left(\delta_x^k \fisheye \left((-1)^{|\beta|}D^\beta\left(\left( \delta_y^k \circledast \delta_0^k\right)\fisheye \widetilde {f_1}\right)\right)\widetilde{}~\, \right)\circledast f_2\\
    &= D^\alpha \left(\delta_x^k \fisheye D^\beta\left(\left(\left( \delta_y^k \circledast \delta_0^k\right)\fisheye \widetilde {f_1}\right)\right)\widetilde{}~\, \right)\circledast f_2\\
    &= D^{\alpha+\beta} \left(\delta_x^k \fisheye \left(\left(\delta_y^k \circledast \delta_0^k\right)\fisheye \widetilde {f_1}\right)\widetilde{}~\, \right)\circledast f_2\\
    &= D^{\alpha+\beta} \left(\left(\delta_x^k \fisheye \left(\left(\delta_y^k \circledast \delta_0^k\right)\fisheye \widetilde {f_1}\right)\widetilde{}~\, \right)\circledast f_2\right)\\
    &= D^{\alpha+\beta}\left(\left(\left(\left(\delta_x^k\circledast\delta_y^k\right)\fisheye\delta_0^k\right)\fisheye f_1\right)\circledast f_2\right)\\
    &= \left(D^{\alpha+\beta}\left(\left(\delta_x^k\circledast\delta_y^k\right)\fisheye\delta_0^k\right)\fisheye f_1\right)\circledast f_2\\
    &= \left(\left(\left(D^\alpha\delta_x^k\circledast D^\beta\delta_y^k\right)\fisheye \delta_0^k\right)\fisheye f_1\right)\circledast f_2.
\end{align*}
This completes the proof of part $(ii)$.
\end{proof}

\section{Proof of Theorem \ref{th: main th r} and Theorem \ref{th: main th r 1}}
In this section we prove the main result in the continuous case. We first prove Theorem \ref{th: main th r 1} and then as a consequence we obtain Theorem \ref{th: main th r}.
\begin{proof}[Proof of Theorem \ref{th: main th r 1}] We present the proof in the following steps.
% We assume \[\left\langle f(x), e_k\right\rangle=\left\langle e_k,f(x)\right\rangle,~f\in S_H.\]    This is used only in step 1.

%\begin{enumerate}
{\textbf{Step 1:}}
In this step, 
   %we find the action of $T$ on the constant function $0,\,e_k$, and $\delta_0^k\in S_H^\prime$. In fact, 
we investigate the image of $T$ on $0$ and $e_k$, $k\in\N$. Using the hypothesis (iii), it is easy to see that $T(0)=0$.

Recall that we have assumed that  $Te_k \in H$. Appealing to the hypothesis (iii), it follows that $Te_k= \alpha_k e_k$ for some $\alpha_k\in \C$. Now our aim is to get some information about $\alpha_k$. In order to do so, first we consider the interesting subclass $S_1$ of $S_H'$, consisting of $\Delta \in S_H'$ such that $(\Delta\blackinwhitesquare f)\otimes g \in S_H$ for all $f,g \in S_H.$ Then it is easy to see that $v\in S_1$ for any $v\in H$. Afterwards, we extend the notion of product for $\Delta'\in S_1$, $\Delta\in S_H'$ and $\phi\in S_H$ as follows. 
$$\left\langle \left(\Delta \otimes \phi\right)\blackinwhitesquare \Delta', f\right\rangle= \left\langle \Delta, \left(\Delta'\blackinwhitesquare \widetilde {\phi}\right)\otimes f \right\rangle,~~f\in S_H.$$

%{\color{red} \st{Now we take $\Delta=e_k=\Delta'$. Then a simple calculations shows that}} 
It is easy to see that $e_k\in S_1$, for all $k\in\N$. Moreover, using the above definition, it follows that $(e_k\otimes\phi)\blackinwhitesquare e_k(x)= \langle \phi(x), e_k\rangle e_k,$ which implies that $(e_k\otimes\phi)\blackinwhitesquare e_k\in S_H$. This suggests that, for $\Delta\in S_H'$, we can define $(\Delta\otimes e_k)\blackinwhitesquare e_k$ as $\langle (\Delta\otimes e_k)\blackinwhitesquare e_k, \phi\rangle = \langle \Delta, (e_k \otimes \phi)\blackinwhitesquare e_k\rangle,$ $\phi\in S_H.$ Then it is easy to see that $(\alpha_k e_k\otimes e_k)\blackinwhitesquare e_k= \alpha_k e_k.$ Further, using the definition of the product, we can show that
\begin{equation}\label{eq: T ONB 1}
\left(\left(\left(e_k\otimes \psi_1\right)\blackinwhitesquare e_k\right)\times f\right)\smblksquare g= \left(e_k \blackinwhitesquare \left(\left(e_k\blackinwhitesquare \psi_1\right)\otimes f\right)\right)\otimes g,\,\,\, \psi_1, f, g\in S_H,
\end{equation}
which in turn implies that
$$
T\left(\left(e_k\otimes \psi_1\right)\blackinwhitesquare e_k\right)= \left(Te_k \otimes T\psi_1\right)\blackinwhitesquare Te_k.
$$
Furthermore, we can show that
$$\left(\left(\left(e_k\otimes e_k\right)\blackinwhitesquare e_k\right)\blackinwhitesquare \psi_1\right)\otimes \psi_2= \left(e_k\blackinwhitesquare\left(\left(e_k\otimes \psi_1\right)\blackinwhitesquare e_k\right)\right)\otimes \psi_2,~~\,\psi_1,\,\psi_2\in S_H.$$ 
We now claim that $T((e_k\otimes e_k)\blackinwhitesquare e_k)= (Te_k\otimes Te_k)\blackinwhitesquare Te_k.$ Indeed, for $\psi_1,\psi_2\in S_H,$ consider
\begin{align*}
\left(T\left(\left(e_k\otimes e_k\right)\blackinwhitesquare e_k\right)\blackinwhitesquare T\psi_1\right)\otimes T\psi_2
&= T\left(\left(\left(\left(e_k\otimes e_k\right)\blackinwhitesquare e_k\right)\blackinwhitesquare \psi_1\right)\otimes \psi_2\right)\\
&= T\left(\left(e_k\blackinwhitesquare\left(\left(e_k\otimes \psi_1\right)\blackinwhitesquare e_k\right)\right)\otimes \psi_2\right)\\
%&= \left(T\Delta \blackinwhitesquare T\left(\left(e_k\otimes \psi_1\right)\blackinwhitesquare e_k\right)\right)\otimes T\psi_2\\
&= \left(Te_k \blackinwhitesquare \left(\left(Te_k\otimes T\psi_1\right)\blackinwhitesquare Te_k\right)\right)\otimes T\psi_2\\
%&= \left(T\Delta \blackinwhitesquare \left(\left(ce_k\otimes T\psi_1\right)\blackinwhitesquare ce_k\right)\right)\otimes T\psi_2\\
&= |\alpha_k|^2 \left(Te_k \blackinwhitesquare \left(\left(e_k\otimes T\psi_1\right)\blackinwhitesquare e_k\right)\right)\otimes T\psi_2\\
        &= |\alpha_k|^2 \left(\left(\left(\left(Te_k\otimes e_k\right)\blackinwhitesquare e_k\right)\blackinwhitesquare T\psi_1\right)\otimes T\psi_2\right)\\
        %&= \left(\left(\left(\left(T\Delta\otimes ce_k\right)\blackinwhitesquare ce_k\right)\blackinwhitesquare T\psi_1\right)\otimes T\psi_2\right)\\
        &= \left(\left(\left(\left(Te_k \otimes Te_k\right)\blackinwhitesquare Te_k\right)\blackinwhitesquare T\psi_1\right)\otimes T\psi_2\right).
    \end{align*}
Now, using the identity, $(e_k\otimes e_k)\blackinwhitesquare e_k=e_k$, it follows that $\alpha_k|\alpha_k|^2=\alpha_k$. In other words, $|\alpha_k|=1$.\\
In what follows, we shall take $\alpha_k=1$ for simplicity, as for any $\alpha_k$ with $|\alpha_k|=1$ the same proof will work. 
       \vspace{10pt}

\textbf{Step 2:} In this step we show that T maps constants to constants. Let $u\in H$. We shall show that $T(u)$ is also a constant. In order to prove this choose $f,g \in S_\C$ such that $\int_{\R^n}f(x)dx=1=\int_{\R^n}g(x)dx$. Let $f_k(x)=f(x)e_k$ and $g_u(x)=g(x)u$. 
Then for $\phi\in S_H$,
\begin{equation}\label{eq: T constant 1}
\left\langle (e_k\fisheye f_k)\circledast g_u,\phi\right\rangle=\left\langle e_k,(\phi\ast \widetilde{g_u}\mdsmblkcircle\widetilde{f_k})\right\rangle=\int_{\R^n}\int_{\R^n}\int_{\R^n}\langle \phi(z), g(z-y)u\rangle f(y-x)dxdydz=\int_{\R^n}\langle\phi(z),u\rangle dz.
\end{equation}
In other words, $(e_k\fisheye f_k)\circledast g_u= u$, as elements of $S_H'$, which implies that
$T(u)= (e_k\fisheye Tf_k)\circledast Tg_u$.
Now, appealing to the hypothesis (iii) we have
$Tf_k=he_k$ for some $h \in S_\C$. 

% Then for $\psi\in S_\C$, proceeding as in \eqref{eq: constant r 1}, we get
% $$
% \left\langle e_k\fisheye Tf_k,\psi\right\rangle=\int_{\R^n}h(x)dx\int_{\R^n}\psi(y)dy=c\int_{\R^n}\psi(y)dy,
% $$
%where $c=\int_{\R^n}h(x)dx$.
% \begin{align*}
%     e_k\fisheye Tf_k(x)= \left\langle \int Tf_k(y)dy, e_k \right\rangle
%     %&= \left\langle \int g(y)e_kdy, e_k \right\rangle\\
%     = \int \left\langle g(y)e_k, e_k \right\rangle dy
%     %&= \int g(y)\left\langle e_k, e_k \right\rangle\\
%     = \int g(y)dy
%     = c \in \C, \,\,\, \text{ (a constant )}.
% \end{align*}
Then for $\phi\in S_H$, proceeding as in \eqref{eq: T constant 1}, we obtain
\begin{align*}
    \left\langle (e_k\fisheye Tf_k)\circledast Tg_u, \phi \right\rangle 
    %&= \left\langle \Delta_c\circledast Tg_u,\phi \right\rangle\\
    %&= \left\langle e_k\fisheye Tf_k, \phi \ast \widetilde{Tg_u}\right\rangle\\
    %&= c\int_{\R^n} \phi \ast \widetilde{Tg_u}(x) dx\\
    %&= c \int \int \left\langle \phi(y), \tau_y \widetilde{Tg_u}(x)\right\rangle dy dx\\
    %&= c \int_{\R^n} \int_{\R^n} \left\langle \phi(y), Tg_u(y-x)\right\rangle dy dx\\
    %&= c \int \left\langle \phi(y), \int Tg_u(y-x)dx\right\rangle dy\\
    %&= \int \left\langle \phi(y), \Bar{c}\int Tg_u(y-x)dx\right\rangle dy\\
    %&= \int_{\R^n} \left\langle \phi(y), \Bar{c}\int_{\R^n} Tg_u(x)dx\right\rangle dy
    = \left\langle v, \phi \right\rangle, 
\end{align*}
where $v=\int_{\R^n}\int_{\R^n} \overline{h(y)}Tg_u(x)dxdy \in H$, a constant. This establishes our claim.\\

%\vspace{10pt}

 % \textbf{Step 3 :} {\color{blue}{\textbf{Action of $T$ on $\delta_x^k$}}}

%\vspace{5pt}

%First, we define what it means by distribution orthogonal to a vector in $H$.\\
% Let $\Delta \in S^\prime_H$. We say that $\Delta$ is {\it orthogonal to $e_k$} if $\Delta(f)=0$ for all $f\in S_H$ satisfying $\left\langle f, e_\ell \right\rangle = 0$ for all $\ell \neq k$, i.e. $\Delta(ge_k)=0$ for all $g \in S_\C$.\\

%{\color{blue}\textbf{$Supp(T\delta_x^k)$ is singleton:}}
{\bf Step 3:} In this step we show that $Supp(T\delta_x^k)$ is singleton and independent of $k$.\\
First we show that $Supp(T\delta_x^k)$ is singleton. If possible let there are two distinct points $y,z$ belonging to $Supp(T\delta_x^k)$. Let $U_y, U_z$ be two open sets containing $y,z$ respectively such that $U_y \cap U_z = \displaystyle \emptyset$. Choose $g_y, g_z \in S_H$ with $Supp(g_y)\subset U_y, Supp(g_z)\subset U_z$ satisfying $T\delta_x^k(g_y)\neq 0, \,\,T\delta_x^k(g_z)\neq 0$. Further, choose $h_y,\, h_z \in S_\C$ having compact support and $h_y\equiv 1$ on $Supp(g_y)$ and $h_z\equiv 1$ on $Supp(g_z)$. Let $h_y^k(t)=h_y(t)e_k$ and $h_z^k(t)= h_z(t)e_k$. 
%So $h_y^k, h_z^k \in S_H$.\\

Then $(T\delta_x^k\blackinwhitesquare g_y)\otimes h_y^k\neq 0$ and $(T\delta_x^k\blackinwhitesquare g_z)\otimes h_z^k\neq 0$.
In fact, if $\left\langle (T\delta_x^k\blackinwhitesquare g_y)\otimes h_y^k, \phi \right\rangle = 0$ for all $\phi \in S_H$, then, taking $\phi= \widetilde{h_y^k}$, we arrive at a contradiction.\\
Let $f_y,\, f_z,\, H_y^k,\, H_z^k \in S_H$ be such that $Tf_y= g_y,\, Tf_z= g_z,\, TH_y^k= h_y^k,\, TH_z^k= h_z^k$. Then, by injectivity of $T$, we have $(\delta_x^k\blackinwhitesquare f_y)\otimes H_y^k\neq 0$ and $ (\delta_x^k\blackinwhitesquare f_z)\otimes H_z^k\neq 0$, which implies that
%\begin{equation}
    $(\delta_x^k\blackinwhitesquare f_y) \neq 0$
%\end{equation}
and
%\begin{equation}
    $(\delta_x^k\blackinwhitesquare f_z) \neq 0$.
%\end{equation}
But
\begin{align*}
 \delta_x^k \blackinwhitesquare f_y(\phi)= \delta_x^k(\phi\smblksquare  f_y)= \left\langle \phi\smblksquare f_y(x), e_k \right\rangle
 %&= \left\langle \phi(x)f_y(x), e_k \right\rangle\\
 = \phi(x)\left\langle f_y(x), e_k \right\rangle,~~\text{ for }\phi\in S_\C.
 \end{align*}
This shows that $\left\langle f_y(x), e_k \right\rangle\neq 0$. Similarly, we can show that $\left\langle f_z(x),e_k \right\rangle \neq 0$.

\noindent
Since, $\delta_x^k$ is orthogonal to $e_k$, $T\delta_x^k$ is orthogonal to $e_k$. So, without loss of generality, we can take $g_y(x)=G_y(x)e_k$ and $g_z(x)=G_z(x)e_k$ for some $G_y,\, G_z \in S_\C$. Since $T^{-1}$ satisfies same properties as $T$, we can take $f_y(x)=F_y(x)e_k$ and $f_z(x)=F_z(x)e_k$ for some $F_y,\, F_z \in S_\C$. Thus, $F_y(x)\neq 0$ and $F_z(x)\neq 0$. Hence $f_y \times \widetilde{f_z}(x)= \left\langle f_y(x), f_z(x)\right\rangle = F_y(x)\overline{F_z(x)}\neq 0$, which implies $g_y\times \widetilde{g_z} \neq 0$. But from the disjointness of the support of $g_y$ and $g_z$, we have $g_y\times \widetilde{g_z}=0$. These are contradictory, hence $Supp(T\delta_x^k)$ is singleton.\\
\par Now we show that $Supp(T\delta_x^k)$ is independent of $k$. The arguments are similar to that of $Supp(T\delta_x^k)$ is singleton. However, we provide an outline here. Let $Supp(T\delta_x^k)=\{y \}$ and $Supp(T\delta_x^\ell)=\{z\}$. Assume that $y\neq z$. Then there exist open sets $U_y, U_z $, containing $y,z$ respectively, such that $U_y\cap U_z=\displaystyle \emptyset$. Choose $g_y^k, g_z^\ell\in S_H$, where $g_y^k(t)=g_y(t)e_k$ and $g_z^\ell(t)=g_z(t)e_\ell$ for some  $g_y,g_z\in S_\C$, $supp(g_y)\subset U_y$, $supp(g_z)\subset U_z$, such that $T\delta_x^k  \left(g_y^k\right) \neq 0$ and $ T\delta_x^\ell \left(g_z^\ell\right) \neq 0$.  
% \begin{equation}\label{eq: singleton 1}
%     \left\langle T\delta_x^k , g_y^k\right\rangle \neq 0
% \end{equation} and 
% \begin{equation}
%     \left\langle T\delta_x^\ell, g_z^\ell\right\rangle \neq 0
% \end{equation}
 \par Again choose $h_y,h_z \in S_\C$ such that $h_y=1$ on $supp(g_y^k)$ and $h_z=1$ on $supp(g_z^\ell)$. Let $h_y^k(t)=h_y(-t)e_k$ and $h_z^\ell(t)=h_z(-t)e_\ell$. Then $T\delta_x^k \blackinwhitesquare g_y^k)\otimes h_y^k\not\equiv 0$ and $(T\delta_x^\ell\blackinwhitesquare g_z^\ell)\otimes h_z^\ell\not\equiv 0$.
Now taking $f_y^k, f_z^\ell, H_y^k, H_z^k \in S_H$ such that $Tf_y^k= g_y^k, Tf_z^\ell= g_z^\ell, TH_y^k= h_y^k, TH_z^\ell=h_z^\ell$ it follows that  $(\delta_x^k \blackinwhitesquare f_y^k)\otimes H_y^k\not\equiv 0$ and $(\delta_x^\ell \blackinwhitesquare f_z^\ell)\otimes H_z^\ell\not\equiv 0$. This implies that $\delta_x^k\blackinwhitesquare f_y^k\not\equiv 0$ and $\delta_x^\ell \blackinwhitesquare f_z^\ell\not\equiv 0$. This in turn implies that $\left\langle f_y^k(x), e_k \right\rangle \neq 0$. Similarly, we obtain $\left\langle f_z^\ell(x), e_\ell \right\rangle \neq 0$.
Moreover, 
\begin{align*}
    \left\langle f_y^k(x), e_k \right\rangle \left\langle e_\ell, f_z^\ell(x) \right\rangle = \left\langle e_\ell , \left\langle e_k, f_y^k(x)\right\rangle \widetilde{f_z^\ell}(-x) \right\rangle=\left\langle e_\ell, \left(e_k\blackinwhitesquare f_y^k\right)\otimes \widetilde{f_z^\ell}(x)\right\rangle\neq 0.
    %= f_z^\ell\times ((e_k\blackinwhitesquare f_y^k)\otimes e_\ell)(x)\neq 0.
\end{align*}
Choose $\phi\in S_H$ such that $\phi(-x)\neq 0$. Then the above equality implies that $e_\ell\blackinwhitesquare\left(\left(e_k\blackinwhitesquare f_y^k\right)\otimes \widetilde{f_z^\ell}\right)\otimes \phi\neq 0$. Now appealing to the hypothesis (ii) and (iv), we obtain
$e_\ell\blackinwhitesquare\left(\left(e_k\blackinwhitesquare g_y^k\right)\otimes \widetilde{g_z^\ell}\right)\otimes T\phi\neq 0.$ But, using the disjointness of support of $g_y$ and $g_z$, it is easy to see that 
$e_\ell\blackinwhitesquare\left(\left(e_k\blackinwhitesquare g_y^k\right)\otimes \widetilde{g_z^\ell}\right)\otimes T\phi= 0$, which is a contradiction. Hence $y=z$.

\vspace{10pt}

\textbf{Step 4:} In this step we investigate the action of $T$ on $\delta_0^k$.\\
We define the function $u:\R^n\to\R^n$ by $u(x)= Supp(T\delta_x^k)$. In view of the previous step, $u$ is well defined.
Since, $Supp(T\delta_x^k)$ is singleton, proceeding as in the classical case (see Theorem 6.25 in \cite{Rudin_FA}),  it follows that there exist $M>0$ and scalars $\gamma_0(x),\,\gamma_\alpha(x),\,1\leq|\alpha|\leq M$ such that
\begin{equation}\label{eq: T delta}
T\delta_x^k= \gamma_0(x)\delta_{u(x)}^k+\sum_{|\alpha|\leq M} \gamma_\alpha(x)D^\alpha \delta_{u(x)}^k.\end{equation}
For $\Delta \in S_H'$, we define $\langle (\Delta \circledast \delta_0^k)\fisheye \delta_0^k, \phi\rangle = \langle \Delta, (\delta_0^k \circledast \phi)\fisheye \delta_0^k\rangle $ for any $\phi \in S_H$. 
%{\color{red}\st{where $(\delta_0^k \circledast \phi)\fisheye \delta_0^k$ is defined as $\langle (\delta_0^k\circledast \phi )\fisheye \delta_0^k, f\rangle= \langle\delta_0^k, (\delta_0^k\fisheye \widetilde {\phi})\circledast f\rangle$, for all $f\in S_H.$ A simple calculation shows that $(\delta_0^k\circledast \phi)\fisheye \delta_0^k(y)=\langle \phi(y), e_k\rangle e_k$, which justifies the above definition.}} 
The validity of this definition follows from the discussion in Section 2 preceding Proposition \ref{pro: delta derivative}.

Let $\Delta=\sum_{|\alpha|\leq M}c_\alpha D^\alpha\delta_x^k$, for some $c_\alpha\in \C$. Then it is easy to see that, $\langle\Delta, (\delta_0^k\circledast \phi)\fisheye \delta_0^k\rangle= \langle \Delta, \langle \phi(\cdot), e_k\rangle e_k\rangle= \langle \Delta , \phi\rangle  $, which in turn implies that 
\begin{equation}\label{eq: s4 delta 0 1}
(\Delta \circledast \delta_0^k)\fisheye \delta_0^k= \Delta.
\end{equation}
 
 %for any $\Delta \perp e_l$ for all $l\neq k.$

We define the subclass $S_2\subset S_H'$ be the set of all $\Delta \in S_H'$ such that $(\Delta\fisheye f)\circledast g \in S_H$ for all $f,g \in S_H.$ Equivalently we can say that $\Delta \in S_2$ iff $\F \Delta\in S_1.$ It is immediate that  $\delta_x^k\in S_2$, for any $k\in\N,\,x\in\R^n$.

Now we extend the notion of convolutions to the class $S_2$. If $\Delta\in S_H'$, $\Delta'\in S_2$, we define $\langle (\Delta\circledast f)\fisheye \Delta', \phi\rangle= \langle \Delta, (\Delta'\fisheye \widetilde {f})\circledast \phi\rangle$.

\par Let $\Delta=\sum_{|\alpha|\leq M}c_\alpha D^\alpha\delta_x^k$. Then using the definition of convolutions and the identity $\left(\Delta\circledast\delta_0^k\right)\fisheye\delta_0^k= \Delta$, it follows that 
$$\left(\left(\left(\Delta\circledast \delta_0^k\right)\fisheye \delta_0^k\right)\fisheye \psi_1\right)\circledast \psi_2= \left(\Delta\fisheye \left(\left(\delta_0^k\circledast \psi_1\right)\fisheye\delta_0^k\right)\right)\circledast\psi_2.$$
% \begin{propn}
%     Let $\Delta\in S_H'$ be such that $\Delta \perp e_l$ for $l\neq k.$ Then we have the following identity:
%     $$\left(\left(\left(\Delta\circledast \delta_0^k\right)\fisheye \delta_0^k\right)\fisheye \psi_1\right)\circledast \psi_2= \left(\Delta\fisheye \left(\left(\delta_0^k\circledast \psi_1\right)\fisheye\delta_0^k\right)\right)\circledast\psi_2.$$
% \end{propn}
% \begin{proof}
%     It can be easily justified using the orthogonality property of $\Delta$ and the identity $\left(\Delta\circledast\delta_0^k\right)\fisheye\delta_0^k= \Delta.$
% \end{proof}
We now aim to show that 
\begin{equation}\label{eq: s4 delta 0 2}
T(\Delta\circledast \delta_0^k)\fisheye \delta_0^k= (T\Delta\circledast T\delta_0^k)\fisheye T\delta_0^k,
\end{equation}
where $\Delta=\sum_{|\alpha|\leq M}c_\alpha D^\alpha\delta_x^k$, for some $c_\alpha\in\C.$ In order to see this, for $\psi_1,\,\psi_2\in S_H$, consider
\begin{align*}
         &\left(T\left(\left(\Delta\circledast\delta_0^k\right) \fisheye \delta_0^k\right)\fisheye T\psi_1\right)\circledast T\psi_2 \\
         &= T\left(\left(\left(\left(\Delta\circledast\delta_0^k\right)\fisheye \delta_0^k\right)\fisheye \psi_1\right)\circledast T\psi_2 \right)\\
         &= T\left(\left(\Delta\fisheye\left(\left(\delta_0^k \circledast \psi_1\right)\fisheye \delta_0^k\right)\right)\circledast\psi_2\right)\\
         &= \left(T\Delta \fisheye T\left(\left(\delta_0^k\circledast \psi_1\right)\fisheye \delta_0^k\right)\right)\circledast T\psi_2\\
         &= \left(T\Delta \fisheye \left(\left(T\delta_0^k\circledast T\psi_1\right)\fisheye T\delta_0^k\right)\right)\circledast T\psi_2\\
         &= \left(T\Delta \fisheye \left(\left(\displaystyle\sum_{|\alpha|\leq M}\gamma_{\alpha}\left(0\right)D^\alpha\delta_{u\left(0\right)}^k\circledast T\psi_1\right)\fisheye \displaystyle\sum_{|\beta|\leq M}\gamma_{\beta}\left(0\right)D^\beta\delta_{u(0)}^k\right)\right)\circledast T\psi_2\\     
         &= \sum_{|\alpha|\leq M} \sum_{|\beta| \leq M} \gamma_\alpha\left(0\right) \gamma_\beta\left(0\right) \left( T \Delta \fisheye \left( \left( D^\alpha \delta_{u(0)}^k \circledast T \psi_1 \right) \fisheye D^\beta \delta_{u(0)}^k \right)\right) \circledast T \psi_2\\
         &= \sum_{|\alpha|\leq M} \sum_{|\beta| \leq M} \gamma_\alpha\left(0\right) \gamma_\beta\left(0\right) \left( T \Delta \fisheye \left ((-1)^{|\alpha|}\left( \left( \delta_{u(0)}^k \circledast D^{\alpha+\beta} T \psi_1 \right) \fisheye \delta_{u(0)}^k \right)\right)\right) \circledast T \psi_2\\ 
         &= \sum_{|\alpha|\leq M} \sum_{|\beta| \leq M} \gamma_\alpha\left(0\right) \gamma_\beta\left(0\right)(-1)^{|\alpha|} \left( T \Delta \fisheye \left (\left( \left( \delta_{u(0)}^k \circledast D^{\alpha+\beta} T \psi_1 \right) \fisheye \delta_{u(0)}^k \right)\right)\right) \circledast T \psi_2\\
         &= \sum_{|\alpha|\leq M} \sum_{|\beta| \leq M} \gamma_\alpha\left(0\right) \gamma_\beta\left(0\right)(-1)^{|\alpha|} \left(\left(\left(\left( T \Delta \circledast \delta_{u(0)}^k\right)\fisheye \delta_{u(0)}^k \right)\fisheye D^{\alpha+\beta} T \psi_1\right) \circledast T \psi_2\right)\\
         &= \sum_{|\alpha|\leq M} \sum_{|\beta| \leq M} \gamma_\alpha\left(0\right) \gamma_\beta\left(0\right)(-1)^{|\alpha|}(-1)^{\alpha+\beta} \left(\left(D^{|\alpha+\beta|}\left(\left( T \Delta \circledast \delta_{u(0)}^k\right)\fisheye \delta_{u(0)}^k \right)\fisheye T \psi_1\right) \circledast T \psi_2\right)\\
         &= \sum_{|\alpha|\leq M} \sum_{|\beta| \leq M} \gamma_\alpha\left(0\right) \gamma_\beta\left(0\right)(-1)^{|\alpha|}(-1)^{|\alpha+\beta|} \left(\left((-1)^{|\beta|}\left(\left( T \Delta \circledast D^\alpha\delta_{u(0)}^k\right)\fisheye D^\beta\delta_{u(0)}^k \right)\fisheye T \psi_1\right) \circledast T \psi_2\right)\\
         &= \left(\left(\left(T\Delta\circledast \sum_{|\alpha|\leq M} \gamma_\alpha(0)D^\alpha\delta_{u(0)}^k\right)\fisheye \sum_{|\beta| \leq M}\gamma_\beta(0)D^\beta\delta_{u(0)}^k\right)\fisheye T\psi_1\right)\circledast T\psi_2\\
         &= \left(\left(\left(T\Delta\circledast T\delta_0^k\right)\fisheye T\delta_0^k\right)\fisheye T\psi_1\right)\circledast T\psi_2.
     \end{align*}
Now, proceeding as in the case of $T\delta_x^k$, we can show that $T^{-1}\delta_x^k$ also has singleton support. Let $p$ be the support of $T^{-1}\delta_0^k$. Then we can write
$$T^{-1}\delta_0^k=\sum_{|\alpha|\leq M'}c_\alpha D^\alpha\delta_p^k,$$
for some $C_\alpha\in\C$ and $M^\prime\in\N$. Thus
$$\delta_0^k= T(T^{-1}\delta_0^k)= T((T^{-1}\delta_0^k\circledast \delta_0^k)\fisheye \delta_0^k)= (\delta_0^k\circledast \delta_{u(0)}^k)\fisheye \delta_{u(0)}^k =\delta^k_{2u(0)}.$$
Comparing the support on both the sides we obtain $u(0)=0$. Therefore we have, $T\delta_0^k= \sum_{|\alpha|\leq M}\gamma_\alpha(0)D^\alpha\delta_0^k.$

%{\color{red}\st{So now we have,}}
Further, by making use of \eqref{eq: s4 delta 0 1} and \eqref{eq: s4 delta 0 2} we have 
$$
\delta_0^k =T\left(\left(T^{-1}\delta_0^k\circledast\delta_0^k\right)\fisheye \delta_0^k\right)=\left(\delta_0^k\circledast T\delta_0^k\right)\fisheye T\delta_0^k= \sum_{|\alpha|\leq M}\sum_{|\beta|\leq M}\gamma_\alpha(0)\gamma_\beta(0) D^{\alpha+\beta}\delta_0^k.
$$
Taking the Fourier transform on both sides, we obtain
$e_k= \sum_{|\alpha|\leq M}\sum_{|\beta|\leq M}\gamma_\alpha(0)\gamma_\beta(0)(2\pi)^{|\alpha+\beta|}(-i)^{|\alpha+\beta|}x^{\alpha+\beta}e_k$, which immediately implies that $\gamma_{\alpha}(0)= 0$ for all $\alpha\neq 0$. So we have $T\delta_0^k= \gamma_0(0)\delta_0^k.$ Now applying $T$ on both sides of the identity $(\delta_0^k\circledast\delta_0^k)\fisheye\delta_0^k=\delta_0^k$, we obtain $(\gamma_0(0))^3=\gamma_0(0)$. In other words, $\gamma_0(0)=\pm 1$. We proceed with $\gamma(0)=1$ as the same technique works even if we take $\gamma_0(0)=-1$.

\vspace{10pt}

\textbf{Step 5:} In this step we show that $u$ is an additive function. 
% In this step we show that the function $u:\R^n\to\R^n$ defined by $u(x)= Supp(T\delta_x^k)$ is additive. \\
% Since, $supp(T\delta_x^k)$ is singleton, proceeding as in the classical case (see Theorem 6.25 in \cite{Rudin_FA}),  it follows that there exist $M>0$ and scalars $\gamma_0(x),\,\gamma_\alpha(x),\,1\leq|\alpha|\leq M$ such that
% \begin{equation}\label{eq: T delta}
% T\delta_x^k= \gamma_0(x)\delta_{u(x)}^k+\sum_{|\alpha|\leq M} \gamma_\alpha(x)D^\alpha \delta_{u(x)}^k.\end{equation}
Since $T(\delta_x^k)$ has the form given in \eqref{eq: T delta}, appealing to the discussions in Section 2, it follows that $(T\delta_x^k \circledast T\delta_y^k)\fisheye \delta_0^k$ is well defined. we now show that, $T((\delta_x^k \circledast \delta_y^k)\fisheye \delta_0^k)= (T\delta_x^k \circledast T\delta_y^k)\fisheye \delta_0^k$.

In order to see this, for $f_1, f_2 \in S_H$, consider
\begin{align*}
\left(T\left(\left(\delta_y^k\circledast \delta_0^k\right)\fisheye \widetilde{\psi_1}\right)~\widetilde{}~\ast Tf_1\right)\mdsmblkcircle Tf_2 &= T\left(\left(\left(\left(\delta_y^k \circledast \delta_0^k\right)\fisheye \widetilde{\psi_1}\right)\widetilde{}~\,\ast f_1\right)\mdsmblkcircle f_2\right)\\
    &= T\left(\left(\delta_y^k\fisheye \left(\left(\delta_0^k\fisheye \psi_1\right)\circledast f_1\right)\right)\circledast f_2\right) \\
    &= \left(T\delta_y^k\fisheye T\left(\left(\delta_0^k \fisheye \psi_1\right)\circledast f_1\right)\right)\circledast Tf_2 \\
    &=  \left(T\delta_y^k\fisheye \left(\left(T\delta_0^k \fisheye T\psi_1\right)\circledast Tf_1\right)\right)\circledast Tf_2\\
    &= \sum_{|\alpha|\leq M}\gamma_\alpha(y)\left(D^\alpha\delta_{u\left(y\right)}^k\fisheye \left(\left(\delta_{0}^k \fisheye T\psi_1\right)\circledast Tf_1\right)\right)\circledast Tf_2\\
    &=\sum_{|\alpha|\leq M}\gamma_\alpha(y)\left(\left(\left(D^\alpha\delta_{u\left(y\right)}^k \circledast \delta_{0}^k\right)\fisheye \widetilde{T\psi_1}\right)\widetilde{}\,\ast Tf_1\right)\mdsmblkcircle Tf_2 \\
    &= \left(\left(\left(T\delta_y^k \circledast T\delta_0^k\right)\fisheye T\widetilde{\psi_1}\right)\widetilde{}\,\ast Tf_1\right)\mdsmblkcircle Tf_2,
\end{align*}
using \eqref{eq: conv asso 3}.
Moreover, for $\psi_1, \psi_2 \in S_H$, we have 
\begin{align*}
    &\left(T\left(\left(\delta_x^k\circledast \delta_y^k\right)\fisheye \delta_0^k\right)\fisheye T\psi_1\right)\circledast T\psi_2 \\&=
    T\left(\left(\left(\left(\delta_x^k \circledast \delta_y^k\right)\fisheye \delta_0^k\right)\fisheye \psi_1\right)\circledast \psi_2\right)\\
    &= T\left(\left(\delta_x^k \fisheye\left(\left(\delta_y^k \circledast \delta_0^k\right)\fisheye \widetilde{\psi_1}\right)\widetilde{}\,\right)\circledast \psi_2\right) \\
    &= \left(T\left(\delta_x^k\right)\fisheye T\left(\left(\delta_y^k \circledast \delta_0^k\right)\fisheye \widetilde{\psi_1}\right)\right)\widetilde{}\,\circledast T\psi_2  \\
    &= \left(T\left(\delta_x^k\right)\fisheye \left(\left(T\delta_y^k \circledast T\delta_0^k\right)\fisheye \widetilde{T\psi_1}\right)\right)\widetilde{}\,\circledast T\psi_2  \\
    % &{\color{blue}\text{will show that, }T\left(\left(\delta_y^k \circledast \delta_0^k\right)\fisheye \widetilde{\psi_1}\right)\right)=  \left(\left(T\delta_y^k \circledast T\delta_0^k\right)\fisheye T\widetilde{\psi_1}\right)\right)}\\
    &= \sum_{|\alpha|\leq M}\sum_{|\beta|\leq N}\gamma_\alpha(x)\gamma_\beta(y)\left(D^\alpha\delta_{u\left(x\right)}^k\fisheye \left(\left(D^\beta\delta_{u\left(y\right)}^k\circledast \delta_{0}^k\right)\fisheye \widetilde{T\psi_1}\right)\right)\widetilde{}\,\circledast T\psi_2\\
    &= \sum_{|\alpha|\leq M}\sum_{|\beta|\leq N}\gamma_\alpha(x)\gamma_\beta(y)\left(\left(\left(D^\alpha\delta_{u\left(x\right)}^k\circledast D^\beta\delta_{u\left(y\right)}^k\right)\fisheye \delta_{0}^k\right)\fisheye T\psi_1\right)\circledast T\psi_2  \\
    &= \left(\left(\left(T\delta_x^k\circledast T\delta_y^k\right)\fisheye T\delta_0^k\right)\fisheye T\psi_1\right) \circledast T\psi_2,
\end{align*}
using \eqref{eq: conv asso 4}, which establishes our claim.
%This shows that $T((\delta_x^k \circledast \delta_y^k)\fisheye \delta_0^k)= (T\delta_x^k \circledast T\delta_y^k)\fisheye T\delta_0^k$.
Thus,
\begin{align*}
    u(x+y) &= Supp(T\delta_{x+y}^k)\\
    &= Supp(T((\delta_x^k \circledast \delta_y^k)\fisheye \delta_0^k))\\
    &= Supp((T\delta_x^k \circledast T\delta_y^k)\fisheye \delta_0^k)) \\
    &=Supp\left(\left(\sum_{|\alpha|\leq M} \gamma_\alpha(x)D^\alpha \delta_{u(x)}^k\circledast \sum_{|\beta|\leq N} \gamma_\beta(y)D^\beta \delta_{u(y)}^k\right) \fisheye \delta_0^k\right)\\
     &= Supp\left( \sum_{|\alpha|\leq M}\sum_{|\beta|\leq N}\gamma_\alpha(x)\gamma_\beta(y)D^{\alpha+\beta}\delta_{u(x)+u(y)}^k \right)\\
     &= u(x)+u(y),
\end{align*}
using Proposition \ref{pro: delta derivative}. 
%{\color{red}\st{$ \left(D^\alpha \delta_x^k\circledast D^\beta \delta_y^k\right)\fisheye \delta_0^k= D^{\alpha+\beta}\delta_{x+y}^k.$}} 
This shows that $u$ is an additive function.

\vspace{10pt}

\textbf{Step 6: } In this step we show that $T\delta_x^k= \delta_{u(x)}^k$.

In view of the equality \eqref{eq: T delta}, we need to show that $\gamma_0(x)=1$ and $\gamma_\alpha(x)=0$ for all $\alpha \neq 0.$ Let $y= \frac{x}{M+1}$. Then
$$ \delta_x^k= \left[\cdots\left[\left[\left[(\delta_y^k\circledast \delta_y^k)\fisheye \delta_0^k\right]\circledast \delta_y^k\right]\fisheye \delta_0^k\right]\cdots \circledast\delta_y^k \fisheye\right]\delta_0^k, $$
where we take $(M-1)$ times appropriate convolution of $\delta_y^k$ and $\delta_0^k$ with $(\delta_y^k\circledast \delta_y^k)\fisheye \delta_0^k$. This implies that
\begin{equation}\label{eq: txk}
 T\delta_x^k= \left[\cdots\left[\left[\left[(T\delta_y^k\circledast T \delta_y^k)\fisheye \delta_0^k\right]\circledast T\delta_y^k\right]\fisheye \delta_0^k\right]\cdots \circledast T\delta_y^k \fisheye\right]\delta_0^k. 
 \end{equation}
Note that, for some $M^\prime\in \N$, we have 
\begin{equation}\label{eq: t delata y}
T\delta_y^k=\gamma_0(y)\delta_0^k+\sum_{1\leq |\alpha|\leq M^\prime}\gamma_\alpha(y)D^\alpha\delta^k_{u(y)}.
\end{equation}
% Let us define the action of $(D^\alpha \delta_x^k\circledast D^\beta \delta_y^k)\fisheye \delta_0^k$ on a $\phi\in S_H$ as 
% $$\left\langle D^\alpha \delta_x^k\circledast D^\beta \delta_y^k)\fisheye \delta_0^k, \phi \right\rangle = \left\langle D^\alpha \delta_x^k, \left(\delta_0^k \circledast \phi\right)\fisheye D^\beta \delta_y^k \right\rangle.$$
% In order to see the above action more explicitly, let us compute $\left(\delta_0^k \circledast \phi\right)\fisheye D^\beta \delta_y^k$ explicitly. For $f\in S_H$, consider,
% \begin{align*}
%     \left\langle \left(\delta_0^k \circledast \phi\right)\fisheye D^\beta \delta_y^k,\psi\right\rangle =\left\langle \delta_0^k, \left(D^\beta \delta_y^k\fisheye \widetilde{\phi}\right)\circledast \psi\right\rangle.
% \end{align*}
% For example when $M=2$, we can write $\delta_x^k= [[(\delta_y^k\circledast \delta_y^k)\fisheye \delta_0^k]\circledast \delta_y^k]\fisheye \delta_0^k$, Now applying $T$ both side, we obtain
% $$T\delta_x^k= [[(T\delta_y^k \circledast T\delta_y^k)\fisheye \delta_0^k]\circledast \delta_0^k]\fisheye \delta_0^k,$$
% And by additivity of $u$, we have $u(x)=3 u(y)$. 
% {\color{blue}We also have to justify $(D^\alpha \delta_x^k\circledast D^\beta \delta_y^k)\fisheye \delta_0^k= D^{\alpha+\beta}\delta_{x+y}^k$.}

Since $u$ is additive, $u(y)= u(x)/(M+1)$. Now substituting $T\delta_y^k$ (as in \eqref{eq: txk}) in \eqref{eq: t delata y} and 
%{\color{red}\st{using $(D^\alpha \delta_x^k\circledast D^\beta \delta_y^k)\fisheye \delta_0^k= D^{\alpha+\beta}\delta_{x+y}^k$}}
appealing to Proposition \ref{pro: delta derivative}, it is easy to see that $\gamma_\alpha(x)= 0$ for all $\alpha\neq 0$.
%Let $T\delta_y^k= \eta_0(y)\delta_{u(y)}^k+\sum_{|\alpha|\leq M'}\eta_\alpha (y)D^\alpha \delta_{u(y)}^k$. Like the classical case, can show that $\eta_\alpha=0$ for all $\alpha\neq 0$. Therefore we have $T\delta_y^k=\eta_0\delta_{u(y)}^k$. From that we can conclude for all $\alpha\neq 0$, we have $\gamma_\alpha=0$. Now we want to show that $\gamma_0=1$. 
\par We now prove that $\gamma_0(x)=1.$
Since, $\gamma_0(0)=1$, it is enough to show that $\gamma(x)$ is independent of $x$. In order to see this observe that $\int_{\R^n} \left\langle e_k, f(z)\right\rangle e_k \, dz\in S_H'$, for $f\in S_H$. Then we have the following identity.
$$\left( \left(\int_{\R^n}  \langle e_k, f(z)\rangle e_k\, dz\right)\fisheye \psi_1\right)\circledast \psi_2= \left( e_k \fisheye f\right) \circledast \left(\left(\delta_x^k \fisheye \psi_1\right)\circledast \psi_2\right),~~\text{ for }\psi_1,\,\psi_2\in S_H.$$
In fact,
\begin{align*}
    &\left\langle\left( \left(\int_{\R^n}  \langle e_k, f(z)\rangle e_k\, dz\right)\fisheye \psi_1\right)\circledast \psi_2, \phi \right\rangle \\
    %= \left\langle \Lambda_{\int_{\R^n}  \langle e_k, f(z)\rangle e_k\, dz}, \left( \phi \ast \widetilde {\psi_2}\right) \mdsmblkcircle \widetilde {\psi_1}\right\rangle\\
    &=\int_{\R^n}  \left\langle \left( \phi \ast \widetilde{\psi_2}\right) \mdsmblkcircle \widetilde{\psi_1}(y), \int_{\R^n}  \left\langle e_k, f(z)\right\rangle e_k\, dz \right\rangle \, dy\\
    %&= \int_{\R^n}  \int_{\R^n}  \left\langle f(z), e_k\right\rangle \left\langle  \left( \phi \ast \widetilde {\psi_2}\right) \mdsmblkcircle \widetilde {\psi_1}(y), e_k \right\rangle \, dz\, dy\\
    %&= \int_{\R^n}  \int_{\R^n}  \left\langle f(z), e_k \right\rangle \left\langle \int_{\R^n} \int_{\R^n}  \left\langle \phi(v), \widetilde {\psi_2}(u-v)\right\rangle \widetilde {\psi_1}(y-u)\, du \, dv, e_k\right\rangle \, dz\, dy\\
    &= \int_{\R^n}  \int_{\R^n}  \int_{\R^n}  \int_{\R^n}  \left\langle f(z), e_k \right\rangle \left\langle \phi(v), \psi_2(v-u)\right\rangle \left\langle \psi_1(u-y), e_k\right\rangle dudvdzdy\\
    &= \int_{\R^n}  \int_{\R^n}  \int_{\R^n}  \int_{\R^n}  \left\langle f(z), e_k \right\rangle \left\langle \phi(v), \psi_2(v-u)\right\rangle \left\langle \psi_1(y), e_k\right\rangle \,du\,dv\,dz\,dy
\end{align*}

On the other hand,
\begin{align*}
    &\left\langle \left( e_k \fisheye f\right) \circledast \left(\left(\delta_x^k \fisheye \psi_1\right)\circledast \psi_2\right), \phi \right\rangle \\&=  \left\langle e_k, \left( \phi \ast \left(\left(\delta_x^k \fisheye \psi_1\right)\circledast \psi_2\right)\widetilde{}~\,\right)\mdsmblkcircle \widetilde {f}\right\rangle\\
    &= \int_{\R^n}  \left\langle \left( \phi \ast \left(\left(\delta_x^k \fisheye \psi_1\right)\circledast \psi_2\right)\widetilde{}~\,\right)\mdsmblkcircle \widetilde {f}(y), e_k \right\rangle\, dy\\
    %&= \int_{\R^n}  \left\langle \int_{\R^n}  \int_{\R^n}  \left\langle \phi(v),\left(\left(\delta_x^k \fisheye \psi_1\right)\circledast \psi_2\right)\widetilde{}~\,\left(u-v\right)\right\rangle \widetilde {f}(y-u)\, du\, dv, e_k\right\rangle \, dy\\
    &= \int_{\R^n}  \int_{\R^n}  \int_{\R^n}  \left\langle \phi(v), \left(\left(\delta_x^k \fisheye \psi_1\right)\circledast \psi_2\right)(v-u)\right\rangle\left\langle f(u-y), e_k\right\rangle \, du\, dv\, dy\\
    &= \int_{\R^n}  \int_{\R^n}  \int_{\R^n}  \left\langle \phi(v), \int_{\R^n}  \left\langle e_k, \psi_1(z-x)\right\rangle \psi_2(v-u-z)\, dz\right\rangle \left\langle f(u-y), e_k\right\rangle\, du\, dv\, dy\\
    %&=\int_{\R^n}  \int_{\R^n}  \int_{\R^n}  \int_{\R^n}  \left\langle \psi_1(z-x), e_k\right\rangle \left\langle \phi(v), \psi_2(v-u-z)\right\rangle\left\langle f(u-y), e_k \right\rangle \, du\, dv\, dy\, dz\\
    &= \int_{\R^n} \int_{\R^n} \int_{\R^n} \int_{\R^n}  \left\langle \psi_1(z), e_k \right\rangle \left\langle \phi(v), \psi_2(v-u)\right\rangle \left\langle f(y), e_k\right\rangle \, du\, dv\, dy\, dz,
\end{align*}
applying change of variables. Hence
\begin{align*}
    \left(T\left(\left(\int_{\R^n}  \left\langle e_k, f(z)\right\rangle e_k \, dz\right)\right) \fisheye T\psi_1\right)\circledast T\psi_2 &= T\left(\left(\left(\int_{\R^n}  \left\langle e_k, f(z)\right\rangle e_k \, dz\right) \fisheye \psi_1\right)\circledast \psi_2\right)\\
    &= T\left( \left( e_k \fisheye f\right)\circledast \left(\left(\delta_x^k \fisheye \psi_1\right) \circledast \psi_2\right)\right)\\
    %&= \left( Te_k \fisheye Tf\right)\circledast \left(\left(T\delta_x^k \fisheye T\psi_1\right) \circledast T\psi_2\right)\\
    &= \left( e_k \fisheye Tf\right)\circledast \left(\left(\gamma(x)T\delta_{u(x)}^k \fisheye T\psi_1\right) \circledast T\psi_2\right)\\
     &= \gamma(x)\left( e_k \fisheye Tf\right)\circledast \left(\left(T\delta_{u(x)}^k \fisheye T\psi_1\right) \circledast T\psi_2\right)\\
     &= \gamma(x)\left(\left( \int_{\R^n}  \left\langle e_k, Tf(z)\right\rangle e_k\, dz\right)\fisheye T\psi_1\right)\circledast T\psi_2,
\end{align*}
which shows that $\gamma(x)$ is independent of $x$. Therefore, $T(\delta_x^k)=\delta_{u(x)}^k.$\\ 
%Hence, $\gamma(x)=\gamma(0)=1.$ 
%Using $T\delta_x^k=\delta_{u(x)}^k$, it follows 

\textbf{Step 7: } In this step we show that there exists $B\in GL_n(\R)$ such that $u(x)=Bx$.

In order to see this observe that $u(0)=0$. Furthermore, $u$ is invertible. In fact, if $u(x)=u(y)$, then $\delta_{u(x)}^k= \delta_{u(y)}^k$, which is same as $T\delta_x^k= T\delta_y^k$. Now, using the injectivity of $T$, we get $x=y.$

Since the hypotheses in the theorem are symmetric with respect to $T$ and $T^{-1}$, the same is true for $T^{-1}$. This, in particular, implies that $u$ is onto.

\par We establish our claim by showing that $u$ is linear. Since, $u$ is additive and invertible it is enough to show that $u$ is locally bounded.

%Then there exist a $B \in GL_n(\R)$ such that $u(x)=Bx$, $x\in\R^n.$

Towards this end, first we show that  
\begin{equation}\label{eq: u lin 1.1}
\left\langle Tf(y), Tg(y)\right\rangle = \left\langle T(f(x)), T(g(x))\right\rangle,~~\text{ for }f,g\in S_H,
\end{equation}
where $y=u(x).$ In view of hypothesis (iv), it is enough to show that
\begin{equation}\label{eq: u lin 1}
\left\langle Tf(y), Tg(-y)\right\rangle = \left\langle T(f(x)), T(g(-x))\right\rangle,
\end{equation}
where $y=u(x)$. Recall that the distribution $\left\langle f(x), g(-x)\right\rangle \delta_x^k$ is defined as
$$\langle \left\langle f(x), g(-x)\right\rangle \delta_x^k, \phi\rangle = \langle f(x), g(-x)\rangle\langle \delta_x^k, \phi \rangle.$$ 

For $\psi_1,\,\psi_2,\,\phi\in S_H$, consider
\begin{align*}
    \left\langle (\left\langle f(x), g(-x)\right\rangle \delta_x^k \blackinwhitesquare \psi_1)\otimes \psi_2, \phi \right\rangle
    %\left\langle ((\delta_x^k \blackinwhitesquare (f\times g))\blackinwhitesquare \psi_1)\otimes \psi_2, \phi \right\rangle &= \left\langle (\delta_x^k \blackinwhitesquare (f\times g)), (\phi \times \psi_2)\smblksquare \psi_1 \right\rangle\\
    &= \left\langle f(x), g(-x)\right\rangle\left\langle \delta_x^k , (\phi \times \psi_2)\smblksquare \psi_1 \right\rangle\\
    %&= \left\langle f(x), g(-x)\right\rangle \left\langle (\phi \times \psi_2)\smblksquare \psi_1 (x), e_k\right\rangle\\
    &= \left\langle f(x), g(-x) \right\rangle \left\langle \phi(x), \psi_2(-x)\right\rangle \left\langle \psi_1(x), e_k \right\rangle\\
    &= \left\langle (\phi \times \psi_2)\smblksquare ((f\times g)\smblksquare \psi_1)(x), e_k \right\rangle\\
    &= \left\langle \delta_x^k, (\phi \times \psi_2)\smblksquare \left((f\times g)\smblksquare \psi_1\right)\right\rangle\\
    &=\left\langle \left(\delta_x^k \blackinwhitesquare \left(\left(f\times g\right)\smblksquare \psi_1\right)\right)\otimes \psi_2, \phi \right\rangle.
    \end{align*}

Hence $(\left\langle f(x), g(-x)\right\rangle \delta_x^k \blackinwhitesquare \psi_1)\otimes \psi_2 = \left(\delta_x^k \blackinwhitesquare \left(\left(f\times g\right)\smblksquare \psi_1\right)\right)\otimes \psi_2$.  Thus,
\begin{align*}
    \left(T\left(\langle f\left(x\right), g\left(-x\right)\right\rangle)\delta_x^k)\blackinwhitesquare T\psi_1\right)\otimes T\psi_2&= T\left(\left(\left(\langle f\left(x\right), g\left(-x\right)\right\rangle)\delta_x^k)\blackinwhitesquare \psi_1\right)\otimes \psi_2\right) \\
    %T\left( \left( \delta_x^k\blackinwhitesquare \left( f\times g\right) \right) \blackinwhitesquare T\psi_1\right) \otimes T\psi_2 &= T\left( \left( \left( \delta_x^k \blackinwhitesquare \left( f\times g\right) \right) \blackinwhitesquare \psi_1\right) \otimes \psi_2\right) \\
    &= T\left( \left( \delta_x^k \blackinwhitesquare \left( \left( f\times g\right) \smblksquare \psi_1\right) \right) \otimes \psi_2\right) \\
    &= \left( T\delta_x^k \blackinwhitesquare T\left( \left( f\times g\right) \smblksquare \psi_1\right) \right) \otimes T\psi_2\\
    &= \left( \delta_{u\left( x\right) }^k \blackinwhitesquare \left( \left( Tf \times Tg\right) \smblksquare T\psi_1\right) \right) \otimes T\psi_2\\
    %&= \left( \left( \delta_{u\left( x\right) }^k \blackinwhitesquare \left( Tf \times Tg\right) \right) \blackinwhitesquare T\psi_1\right) \otimes T\psi_2 \\
    %&= \left( \left( T\delta_x^k \blackinwhitesquare \left( Tf \times Tg\right) \right) \blackinwhitesquare T\psi_1\right) \otimes T\psi_2
    &= \left(\left\langle Tf\left(u\left(x\right)\right), Tg\left(-u\left(x\right)\right)\right\rangle \delta_{u\left(x\right)}^k\blackinwhitesquare T\psi_1\right)\otimes T\psi_2.
\end{align*}
This shows that

\begin{equation}\label{eq: u lin2}
T\left(\langle f\left(x\right), g\left(-x\right)\right\rangle)\delta_x^k)= \left\langle Tf\left(y\right), Tg\left(-y\right)\right\rangle \delta_{y}^k,
\end{equation}
where $u(x)=y.$
%$$T(\delta_x^k\blackinwhitesquare(f\times g))= T\delta_x^k \blackinwhitesquare (Tf\times Tg).$$

% In order to show that $u$ is locally bounded, let $A\subset\R^n$ be bounded. Let $B\subset \R^n$ be an open set containing $A$. Choose $f\in S_\C$ such that $f=1$ on $A$ and $f=0$ on $B^c$. Let $f_k(x)=f(x)e_k$. Then
% $$
% \|Tf_k(u(x))\|^2=\langle Tf_k(u(x)), Tf_k(u(x))\rangle=\langle T(f_k(x)), T(f_k(x))\rangle=1,
% $$
% for $x\in A$. Since, $Tf_k$ is a Schwartz function, $u(A)$ is bounded. In other words, $u$ is locally bounded. 
 
% % {\color{teal}
% % \textbf{Final Step:} {\color{blue} \textbf{Defining $\beta$ and Properties of $\beta$.}}

% % {\bf Final step:}
% % For $c,d\in H$ let us define the distribution $(c\blackinwhitesquare d)\otimes e_1$ as $(c\blackinwhitesquare d)\otimes e_1 = \left\langle c, d \right\rangle e_1$
% % \vspace{5pt}

% So, from the previous step, we have, 
% $${\color{blue} \left\langle Tf(y), Tg(-y) \right\rangle = \left\langle T(f(x)), T(g(-x))\right\rangle= \beta(\left\langle f(x), g(-x\right\rangle)) }$$

Note that if we can show that 
\begin{equation}\label{eq: u lin 2.1}
    T\left(\langle f\left(x\right), g\left(-x\right)\right\rangle)\delta_x^k)=\left\langle T(f(x)),T(g(-x))\right\rangle T\delta_x^k,
\end{equation}
then the equality \eqref{eq: u lin 1} follows from \eqref{eq: u lin2}.
% \begin{equation}

% \end{equation}
%  Then \eqref{eq: u lin 1} follows from \eqref{eq: u lin2}.
 
In order to show \eqref{eq: u lin 2.1}, we first establish the following identities.  
\begin{align}
\label{eq: u lin 3} T(\langle \phi, \alpha e_k\rangle d)&= \langle T\phi,  T(\alpha e_k)\rangle Td,\\    
\label{eq: u lin 3.1} T(\left\langle c, d \right\rangle \alpha e_k)&= \left\langle Tc, Td \right\rangle T(\alpha e_k),\\
\label{eq: u lin 3.2} T(\left\langle ce_k, e_k\right\rangle\delta_x^k)&= \left\langle T(ce_k), e_k \right\rangle T\delta_x^k,
\end{align}
where $\phi\in S_H,\,c, d\in H,\, \alpha\in\C$ and $\langle \phi, e_k\rangle d(x)= \langle \phi(x), e_k\rangle d$.
We prove the above identities below.
\par For $\psi_1, \psi_2\in S_H$, consider 
\begin{align*}
    (\langle \phi, \alpha e_k\rangle d\times \psi_1)\smblksquare \psi_2(x) =\left\langle \left\langle \phi, \alpha e_k\right\rangle d(x), \psi_1(-x)\right\rangle \psi_2(x)
    = \overline{\alpha} \langle \phi(x), e_k\rangle \langle d, \psi_1(-x)\rangle \psi_2(x).
\end{align*}
Further, for $f\in S_H$
\begin{align*}
    \langle (d\blackinwhitesquare ((\alpha e_k\blackinwhitesquare \phi)\otimes \psi_1))\otimes \widetilde {\psi_2}, f\rangle &= \langle d, (f\times \widetilde {\psi_2})\smblksquare ((\alpha e_k\blackinwhitesquare \phi)\otimes \psi_1)\rangle\\
    &= \int_{\R^n}  \langle (f\times \widetilde {\psi_2})\smblksquare ((\alpha e_k\blackinwhitesquare \phi)\otimes \psi_1)(x), d\rangle \, dx\\
    %&= \int_{\R^n}  \left\langle \left\langle f(x), \widetilde {\psi_2}(x)\right\rangle \left\langle e_1, \phi(x)\right\rangle \psi_1(-x), d\right\rangle\, dx\\
    &=  \int_{\R^n}  \left\langle f(x), \overline{\alpha}\left\langle \phi(x), e_k\right\rangle \left\langle d, \psi_1(-x)\right\rangle \psi_2(x)\right\rangle\, dx,
\end{align*} 
using $(\alpha e_k\blackinwhitesquare \phi)\otimes \psi_1(x)=\alpha \langle e_k, \phi(x)\rangle \psi_1(-x).$ 
This shows that
$$\langle \phi, \alpha e_k\rangle \times \psi_1)\smblksquare \psi_2= (d\blackinwhitesquare ((\alpha e_k\blackinwhitesquare \phi)\otimes \psi_1))\otimes \widetilde {\psi_2}.$$
Hence
\begin{align*}
    (T(\langle \phi, \alpha e_k \rangle d)\times T\psi_1)\smblksquare T\psi_2 &= T((\langle \phi, \alpha e_k \rangle d\times \psi_1)\smblksquare \psi_2)\\
    &= T((d\blackinwhitesquare ((\alpha e_k \blackinwhitesquare \phi)\otimes \psi_1))\otimes \widetilde {\psi_2})\\
    &= (Td\blackinwhitesquare ((T(\alpha e_k) \blackinwhitesquare T\phi)\otimes T\psi_1))\otimes T\widetilde {\psi_2}\\
    &= (Td\blackinwhitesquare ((T(\alpha e_k) \blackinwhitesquare T\phi)\otimes T\psi_1))\otimes \widetilde{T\psi_2}\\
    &= (\langle T\phi, T(\alpha e_k) \rangle Td\times T\psi_1)\smblksquare \widetilde{\widetilde{T\psi_2}}\\
    &= (\langle T\phi, T(\alpha e_k) \rangle Td\times T\psi_1)\smblksquare T\psi_2,
\end{align*}
which proves \eqref{eq: u lin 3}.
%Thus $T(\langle \phi, e_k \rangle d)= \langle T\phi, Te_k \rangle Td.$

Now recall that for $c,d\in H$, the distribution $\langle c,d\rangle \alpha e_k $ is defined by $$\left\langle \alpha \left\langle c,d \right\rangle e_k , \phi \right\rangle= \int_{\R^n}\left\langle \phi(x), \alpha \left\langle c,d\right\rangle \, e_k  \right\rangle dx,$$
where $\phi\in S_H$.
For $\psi_1,\psi_2, \phi \in S_H,$ we have
$$
\left\langle (\left\langle c, d\right\rangle \alpha e_k \blackinwhitesquare \psi_1)\otimes \psi_2, \phi \right\rangle=\left\langle \alpha \langle c,d\rangle e_k,(\phi\times \psi_2)\smblksquare \psi_1\right\rangle=\overline{\alpha \left\langle c,d\right\rangle} \int_{\R^n} \left\langle  \left\langle \phi(x), \psi_2(-x)\right\rangle \psi_1(x), e_k \right\rangle\, dx.
$$
% \begin{align*}
%     \left\langle (\left\langle c, d\right\rangle e_k \blackinwhitesquare \psi_1)\otimes \psi_2, \phi \right\rangle &= \left\langle \left\langle c,d \right\rangle e_k , (\phi\times \psi_2)\smblksquare \psi_1\right\rangle\\
%     &= \left\langle e_k , \left\langle c,d \right\rangle (\phi\times \psi_2)\smblksquare \psi_1\right\rangle\\
%     &= \int \left\langle \left\langle c,d \right\rangle (\phi\times \psi_2)\smblksquare \psi_1(x), e_k \right\rangle\, dx\\
%     &=\int \left\langle \left\langle c,d\right\rangle \left\langle \phi(x), \psi_2(-x)\right\rangle \psi_1(x), e_k \right\rangle\, dx\\
%     &= \int \left\langle \phi(x), \left\langle d,c\right\rangle\left\langle e_k , \psi_1(x)\right\rangle \psi_2(-x)\right\rangle \, dx
% \end{align*}
% Hence $(\left\langle c, d\right\rangle e_k \blackinwhitesquare \psi_1)\otimes \psi_2(x)= \left\langle d,c\right\rangle\left\langle e_k , \psi_1(x)\right\rangle \psi_2(-x)$ pointwise. 
On the other hand,
$$
\left\langle \left(c \blackinwhitesquare \left(\left \langle \psi_1, \alpha e_k \right \rangle d\right)\right)\otimes \psi_2, \phi \right\rangle = \left\langle c, (\phi \times \psi_2)\smblksquare (\langle \psi_1, \alpha e_k \rangle d)\right\rangle=\overline{\alpha \left\langle c,d\right\rangle} \int_{\R^n} \langle \phi(x), \psi_2(-x)\rangle\langle \psi_1(x), e_k \rangle \, dx,
$$
% \begin{align*}
%     \left\langle \left(d\blackinwhitesquare \left(\left \langle \psi_1, e_k \right \rangle c\right)\right)\otimes \psi_2, \phi \right\rangle &= \left\langle d, (\phi \times \psi_2)\smblksquare (\langle \psi_1, e_k \rangle c)\right\rangle\\
%     &= \int \left\langle(\phi \times \psi_2)\smblksquare (\langle \psi_1, e_k \rangle c)(x), d\right \rangle\, dx\\
%     &= \int \langle \phi(x), \psi_2(-x)\rangle\langle \psi_1(x), e_k \rangle \langle c,d\rangle\, dx\\
%     &= \int \left\langle \phi(x), \left\langle d,c \right\rangle \left\langle e_k , \psi_1(x)\right\rangle \psi_2(-x)\right\rangle \, dx\\
% \end{align*}
% So, pointwise we can write $\left(d\blackinwhitesquare \left(\left \langle \psi_1, e_k \right \rangle c\right)\right)\otimes \psi_2(x)=\left\langle d,c \right\rangle \left\langle e_k , \psi_1(x)\right\rangle \psi_2(-x)$

which shows that $(\left\langle c, d\right\rangle e_k \blackinwhitesquare \psi_1)\otimes \psi_2= \left(c \blackinwhitesquare \left(\left \langle \psi_1, e_k \right \rangle d\right)\right)\otimes \psi_2. $
Therefore,
\begin{align*}
    (T(\langle c,d\rangle e_k )\blackinwhitesquare T\psi_1)\otimes T\psi_2 &= T((\langle c,d\rangle e_k \blackinwhitesquare \psi_1)\otimes \psi_2)\\
    &= T((c \blackinwhitesquare (\langle \psi_1, e_k \rangle d))\otimes \psi_2)\\
    &= (Tc \blackinwhitesquare T(\langle \psi_1, e_k \rangle d))\otimes T\psi_2\\
    &= (Tc\blackinwhitesquare (\langle T\psi_1, Te_k \rangle Td))\otimes T\psi_2\\
    &= (Tc\blackinwhitesquare (\langle T\psi_1, Te_k \rangle Td))\otimes T\psi_2\\
    &= (\langle Tc,Td\rangle e_k \blackinwhitesquare T\psi_1)\otimes T\psi_2\\
    &= (\langle Tc,Td\rangle Te_k \blackinwhitesquare T\psi_1)\otimes T\psi_2,
\end{align*}
which establishes \eqref{eq: u lin 3.1}.
%This shows that $T(\langle c,d\rangle e_k )= \langle Tc,Td\rangle Te_k $.

Now for any $c\in\C$, $\langle ce_k,e_k\rangle\delta_x^k$ defines a distribution.

% So,
% \begin{align*}
%      \left\langle \delta_x^k \otimes(ce_k \blackinwhitesquare e_k, \phi \right\rangle & := \left\langle \delta_x^k, (ce_k \otimes \widetilde{\phi}) \blackinwhitesquare  e_k \right\rangle\\
%      &=  \left\langle (ce_k \otimes \widetilde{\phi})\blackinwhitesquare  e_k (x), e_k \right\rangle \\
%      &=  \left\langle \left\langle \phi(x), e_k \right\rangle ce_k , e_k \right\rangle \\
%      &= c\left\langle \phi(x), e_k\right\rangle\\
%      &= \left\langle \left\langle ce_k, e_k \right\rangle \delta_x^k, \phi \right\rangle
% \end{align*}
For $\psi_1,\,\psi_2,\,\phi \in S_H$, consider
\begin{align*}
    \left\langle \left( \left( \langle ce_k,e_k\rangle\delta_x^k\right)\blackinwhitesquare \psi_1\right)\otimes \psi_2, \phi\right\rangle 
    %&= \left\langle (c\delta_x^k \blackinwhitesquare \psi_1)\otimes \psi_2, \phi \right\rangle \\
    = c\left\langle \delta_x^k, (\phi \times \psi_2)\smblksquare \psi_1 \right\rangle
    %&= \left\langle \delta_x^k, c(\phi \times \psi_2)\smblksquare \psi_1 \right\rangle\\
    %&= \left\langle c(\phi \times \psi_2)\smblksquare \psi_1 (x), e_k \right\rangle \\
    %&= c \left\langle \left\langle \phi(x), \psi_2(-x) \right\rangle \psi_1(x), e_k \right\rangle\\
    = c \left\langle \phi(x), \psi_2(-x) \right\rangle \left\langle \psi_1(x), e_k \right\rangle.
\end{align*}
Further, it is easy to see that
\begin{align*}
    \left\langle  \left( \delta_x^k \blackinwhitesquare \left( \left\langle \psi_1, e_k \right\rangle ce_k \right) \right) \otimes \psi_2, \phi\right\rangle = \left\langle \delta_x^k, \left( \phi \times \psi_2\right) \smblksquare \left( \left\langle \psi_1, e_k \right\rangle ce_k \right) \right\rangle 
    %&= \left\langle \left( \phi \times \psi_2\right) \smblksquare \left( \left\langle \psi_1, e_k \right\rangle ce_k \right) \left( x\right) , e_k \right\rangle\\
    %&= \left\langle \left\langle \phi\left( x\right) , \psi_2\left( -x\right) \right\rangle \left\langle \psi_1\left( x\right) , e_k \right\rangle ce_k, e_k \right\rangle \\
    %&= \left\langle \phi\left( x\right) , \psi_2\left( -x\right)  \right\rangle \left\langle \psi_1\left( x\right) , e_k \right\rangle c \left\langle e_k, e_k \right\rangle\\
    = c \left\langle \phi\left( x\right) , \psi_2\left( -x\right) \right\rangle \left\langle \psi_1\left( x\right) , e_k \right\rangle. 
\end{align*}

So we have $\left( \left( \langle ce_k,e_k\rangle\delta_x^k\right)\blackinwhitesquare \psi_1\right)\otimes \psi_2=\left( \delta_x^k \blackinwhitesquare \left( \left( \psi_1\blackinwhitesquare e_k \right) \otimes ce_k \right) \right) \otimes \psi_2 .$

Hence,
\begin{align*}
    \left(T\left(\left( \langle ce_k,e_k\rangle\delta_x^k\right)\right)\blackinwhitesquare T\psi_1\right)\otimes T\psi_2 &= T\left(\left(\left( \langle ce_k,e_k\rangle\delta_x^k\right)\blackinwhitesquare \psi_1\right)\otimes \psi_2\right)\\
    &= T\left(\left(\delta_x^k\blackinwhitesquare \left( \left\langle \psi_1, e_k \right\rangle ce_k \right)\right)\otimes \psi_2\right)\\
    &= \left(T\delta_x^k \blackinwhitesquare T\left( \left\langle \psi_1, e_k \right\rangle ce_k \right)\right)\otimes T\psi_2\\
    %&=\textcolor{red}{  \left(T\delta_x^k \blackinwhitesquare \left(\left( T\psi_1\blackinwhitesquare Te_k \right) \otimes T\left(ce_k\right)\right)\right)\otimes T\psi_2}\\
    &= \left(T\delta_x^k \blackinwhitesquare \left( \left\langle T\psi_1, e_k \right\rangle T(ce_k) \right)\right)\otimes T\psi_2\\
    %&= \left(\left(\langle de_k,e_k\rangle\delta_y^k\right)\blackinwhitesquare T\psi_1\right) \otimes T\psi_2\\
    &= \left(\left(\langle T\left(ce_k\right),e_k\rangle T\delta_x^k\right)\blackinwhitesquare T\psi_1\right) \otimes T\psi_2,
\end{align*}
which shows \eqref{eq: u lin 3.2}.
%This shows that, $T(\left\langle ce_k, e_k\right\rangle\delta_x^k)= \left\langle T(ce_k), e_k \right\rangle T\delta_x^k$.\\
% \textbf{Associativity Justification:}
% We, want to show that, $$\left( \left( \delta_x^k \otimes \left( ce_k \blackinwhitesquare e_k \right)\right)\blackinwhitesquare \psi_1\right)\otimes \psi_2 = (\delta_x^k \blackinwhitesquare (ce_k \otimes(e_k \blackinwhitesquare \psi_1)))\otimes \psi_2$$

% So we have 

Now using \eqref{eq: u lin 3}, \eqref{eq: u lin 3.1} and \eqref{eq: u lin 3.2}, we have
\begin{align*}
    \left\langle Tf\left(y\right), Tg\left(-y\right)\right\rangle \delta_y^k %&= T\delta_x^k \blackinwhitesquare \left(Tf \times Tg\right)\\
    %&= T\left(\delta_x^k \blackinwhitesquare \left(f\times g\right)\right)\\
    &= T\left(\left\langle f\left(x\right), g\left(-x\right) \right\rangle \delta_x^k\right)\\
    &= T\left(\left\langle \left\langle f\left(x\right), g\left(-x\right)\right\rangle e_k, e_k \right\rangle \delta_x^k\right)\\
    &= \left\langle T\left(\left\langle f\left(x\right), g\left(-x\right)\right\rangle e_k\right), Te_k \right\rangle T\delta_x^k \\
    &= \textcolor{black}{\left\langle \left\langle T\left(f\left(x\right)\right), T\left(g\left(-x\right)\right)\right\rangle e_k, e_k \right\rangle \delta_y^k} \\
    &= \left\langle T\left(f\left(x\right)\right), T\left(g\left(-x\right)\right)\right\rangle \delta_y^k,
\end{align*}
which proves \eqref{eq: u lin 2.1}.
\par In order to show that $u$ is locally bounded, let $W\subset\R^n$ be bounded. Let $V\subset \R^n$ be an open set containing $W$. Choose $f\in S_\C$ such that $f=1$ on $W$ and $f=0$ on $V^c$. Let $f_k(x)=f(x)e_k$. Then
$$
\|Tf_k(u(x))\|^2=\langle Tf_k(u(x)), Tf_k(u(x))\rangle=\langle T(f_k(x)), T(f_k(x))\rangle=1,
$$
using \eqref{eq: u lin 1.1}, for $x\in W$. Since, $Tf_k$ is a Schwartz function, $u(W)$ is bounded. In other words, $u$ is locally bounded. 
Hence $u$ is linear and thus, there exists $B\in GL_n(\R)$ such that $u(x)=Bx$.

\vspace{10pt}

\item \textbf{Step 8: } In this step we show that the map $\beta: \C \to \C$ defined by $$ \beta(c)= \left\langle T(ce_1), e_1\right\rangle$$
can be written as 
$$
\beta(c)= \left(\frac{c}{|c|}\right)^k|c|^\alpha, \text{ for some } Re(\alpha)>0, k\in \{\pm 1\}.
$$
In order to obtain the above form it is enough to show that $\beta$ is a continuous multiplicative bijection on $\C$.
Observe that for $c,d \in H$,
\begin{align}\label{eq: beta cd Tcd}
    \beta(\left\langle c,d \right\rangle )= \left\langle T(\left\langle c,d \right\rangle e_1), e_1 \right\rangle= \left\langle \left\langle Tc, Td \right\rangle e_1, e_1 \right\rangle = \left\langle Tc, Td \right\rangle,
\end{align}
using \eqref{eq: u lin 3.1}.

%\textbf{Properties of $\beta$ : }\\
% \begin{itemize}
%     \item [(i)] $\beta$ is bijective
%     \item [(ii)] $\beta$ is additive
%     \item [(iii)] $\beta$ is multiplicative
%     %\item [(iv)] $\beta$ is continuous 
% \end{itemize}
% \begin{proof}
%We now show that the function $\beta$ is an additive and multiplicative bijection which maps $\R$ onto $\R$. These properties of $\beta$ will conclude that either $\beta(c)=c$ or $\beta(c)=\overline{c}$. This in turn will lead to the desired conclusion in view of \eqref{eq: u lin 1} and making use of the hypothesis (iv).

%\par So, we are now left with proving the specified properties of $\beta$. The additivity of $\beta$ follows from that of the operator $T$. 
%\begin{align*}
%    \beta(c)\beta(d)&= \left\langle T(ce_1),e_1 \right\rangle \left\langle T(de_1), e_1\right\rangle = {\color{teal}\left\langle \left\langle T(ce_1),e_1 \right\rangle T(de_1), e_1 \right\rangle= \left\langle T(\left\langle ce_1, e_1\right\rangle de_1), e_1 \right\rangle} = \left\langle T(cde_1), e_1 \right\rangle\\
%    &= \beta(cd)
%\end{align*}
%
%====
%{\color{black}
In order to show that $\beta$ is multiplicative, observe that $T(\bar{c}e_1)= \overline{c'}e_1$, where $T(ce_1)=c'e_1$. In fact,
\begin{align*}
    T(\bar{c}e_1) &= T(\langle e_1, ce_1\rangle e_1)= \langle Te_1, T(ce_1)\rangle Te_1= \langle e_1, c'e_1\rangle e_1= \overline{c'}e_1,
    %=\overline{c'e_1}=\overline{T(ce_1)}.
\end{align*}
using \eqref{eq: u lin 3.1}. Thus, 
\begin{align*}
\beta(cd)=\langle T(cde_1),e_1\rangle=\left\langle T\left(\langle ce_1, \bar{d}e_1\rangle e_1\right), e_1\right\rangle=\langle T(ce_1),T(\bar{d}e_1)\rangle= \langle c'e_1, \overline{d'}e_1\rangle&= \left\langle c'e_1, e_1\right\rangle \left\langle d'e_1, e_1\right\rangle\\
&=\beta(c)\beta(d),
\end{align*} 
using \eqref{eq: u lin 3.1}, where $T(de_1)=d^\prime e_1$.
%\begin{align*}
%    \beta(cd) &= T(cde_1)\\
%    &= T(\langle ce_1, \overline{d}e_1\rangle e_1)\\
%    &= \langle T(ce_1), T(\overline{d}e_1)\rangle Te_1\\
%    &= \langle c'e_1, \overline{T(de_1)}\rangle e_1\\
%    &=\langle c'e_1, \overline{d'}e_1\rangle\\
%    &=c'd'e_1
%\end{align*}
%
%\begin{align*}
%    \beta(c)\beta(d) &= \left\langle T(ce_1), e_1\right\rangle\left\langle T(de_1), e_1\right\rangle\\
%    &= \left\langle \left\langle T(ce_1), e_1\right\rangle T(de_1), e_1\right\rangle\\
%    &=\left\langle \left\langle c'e_1, e_1\right\rangle d'e_1, e_1\right\rangle\\
%    &= c'd'e_1
%\end{align*}
%Hence, we have proved that $\beta$ is multiplicative.

%
%====
Let $c,d \in \C$ be such that $\beta(c)=\beta(d)$. Then $\left\langle T(ce_1), e_1 \right\rangle = \left\langle T(de_1), e_1 \right\rangle$. Since $T(ce_1)$ and $T(de_1)$ have components along $e_1$ only, $T(ce_1)=T(de_1)$. Hence from the injectivity of $T$, we get the injectivity of $\beta$. The ontoness of $\beta$ follows from that of $T$.
\par 
We now show that $\beta$ is continuous. As in \cite{Alesker_paper}, we obtain the required continuity by making use of the fact that $T$ leaves $S_H$ invariant. Since $\C^*= S^1\times \R_{>0}$ and $\beta(r e^{i\theta})= \beta(r)\beta(e^{i\theta})$, it suffices to prove that $\beta|_{S^1}$ and $\beta|_{\R_{>0}}$ are continuous, where $S^1$ denotes the unit sphere in $\R^n$.

For $a\in [-2\pi, 2\pi]$ construct $f,g\in S_{\C}$ such  that for $x=(a,0,\cdots ,0), f(x)= e^{ia/2}, g(x)= e^{-ia/2}$ and both $f,g$ have compact support. Let $f_k(x)= f(x)e_k, g_k(x)=g(x)e_k$. Then for $y= B(a,0,\cdots ,0)$, appealing to \eqref{eq: u lin 1.1} and \eqref{eq: beta cd Tcd} we have $\langle Tf_k(y), Tg_k(y)\rangle= \beta\left(\left\langle f_k(B^{-1}y), g_k(B^{-1}y)\right\rangle\right)= \beta(e^{ia}).$ Now, since $Tf_k$ and $Tg_k$ are continuous, $\beta|_{S^1}$ must be continuous.

Now fix $r_1>r_0>0$, construct $f,g \in S_{\C}$ such that for $x=(a,0,\cdots ,0)$ with $a\in [r_0,r_1]$ we have $f(x)=\sqrt{a},\, g(x)=\sqrt{a}$ and $f,g$ both have compact support. Let $f_k(x)= f(x)e_k, g_k(x)=g(x)e_k$. Then for $y=B(a,0,\cdots ,0)$, where  $a\in [r_0,r_1]$, we have $\langle Tf_k(y), Tg_k(y)\rangle =  \beta(a).$ Since $Tf_k,\, Tg_k$ are continuous, we have $\beta|_{[r_0,r_1]}$ is continuous. As $r_0, r_1$ are arbitrary we have $\beta|_{\R_{>0}}$ is continuous. Using sequential criteria we can easily check that $\beta$ is continuous at $0$.

\vspace{10pt}

{\bf Step 9:} Finally in this step we obtain the desired equality. Note that, in the previous step, in view of \eqref{eq: u lin 1.1} and \eqref{eq: beta cd Tcd} we have shown that \begin{equation}\label{eq: beta T final}
\langle Tf(y), Tg(y)\rangle= \left( \frac{\langle f(B^{-1}y), g(B^{-1}y)\rangle}{|\langle f(B^{-1}y), g(B^{-1}y)\rangle|}\right)^k\left|\langle f(B^{-1}y), g(B^{-1}y)\rangle\right|^\alpha,
\end{equation}
where $Re(\alpha)>0$ and $k=\pm 1$. So, it is enough to show that $\alpha=1$. 
\par In order to see this take $\epsilon>0$ be arbitrary. Then choose $\delta>0$ be such that the volume of $B([-1/2-\delta, 1/2+\delta]^n-[-1/2+\delta, 1/2-\delta]^n)< \epsilon$. We now consider a function $f\in S_{\C}$ such that 
$$f(x)=\begin{cases}
    1 & x\in [-1/2+\delta, 1/2-\delta]^n,\\
    0 & x\notin [-1/2-\delta, 1/2+\delta]^n,
\end{cases} 
$$
$0\leq f(x)\leq 1$ and $\int_{\R^n}f(x)\, dx=1.$ Set $f_1(x)=f(x)e_1$. 
% The for $\phi \in S_H$, 
% \begin{align*}
%     \langle (e_1\fisheye f_1)\circledast f_1, \phi\rangle &= \langle e_1, (\phi \circledast f_1)\fisheye f_1\rangle\\
%     &= \int_{\R^n} \langle (\phi \ast f_1)\mdsmblkcircle f_1(x), e_k \rangle\, dx\\
%     &= \int_{\R^n}\int_{\R^n} \int_{\R^n} \left\langle \left\langle \phi(z), f_1(y-z)\right\rangle f_1(x-y), e_1\right\rangle\, dz\, dy\, dx\\
%     &= \int_{\R^n}\int_{\R^n} \int_{\R^n} \left\langle \phi(z), f_1(y-z)\right\rangle f(x-y)\, dz\, dy \, dx \\
%     &= \int_{\R^n}\int_{\R^n} \int_{\R^n} \left\langle \phi(z), f(x-y) f_1(y-z)\right\rangle\, dz\, dy\, dx.
% \end{align*}
% So, pointwise we can write, $(e_1\fisheye f_1)\circledast f_1(z)= \int_{\R^n}\int_{\R^n}f(x-y)f_1(y-z)\, dy\, dx= \left(\int f(y)\, dy \int f(x)\, dx\right)e_1= e_1.$ 
Then it is easy to see that $(e_1\fisheye f_1)\circledast f_1=e_1.$ Now let $Tf_1(x)= g(x)e_1$, for some $g\in S_{\C}$ (this is possible due to the hypothesis (iii)). Then a straightforward computation shows that
$(e_1\fisheye Tf_1)\circledast Tf_1(z)= |\int_{\R^n} g(x)\, dx|^2 e_1.$ This immediately implies that $|\int_{\R^n} g(x)\, dx|= 1$, hence $\|\int_{\R^n} Tf_1(x)\, dx\|=1.$

On the other hand, $\langle Tf_1(y), Tf_1(y)\rangle=  |\langle f_1(B^{-1}y), f_1(B^{-1}y)\rangle|^\alpha $. Thus $|g(y)|= \left(f(B^{-1}y)\right)^\alpha.$ Moreover, the choice of $f_1$ gives 
\begin{align*}
\|Tf_1(y)\|= 
    \begin{cases}
      1 & y\in  B([-1/2+\delta, 1/2-\delta]^n)  \\
      0 & y\notin B([-1/2-\delta, 1/2+\delta]^n),
    \end{cases}
\end{align*}
%$||Tf_1(y)||= 1$ on $B([-1/2+\delta, 1/2-\delta]^n)$, $0$ outside $B([-1/2-\delta, 1/2+\delta]^n)$ and 
and $0\leq \|Tf_1(y)\|\leq 1.$ Then one can show that $\left|\int_{\R^n}\|Tf_1(y)\|\, dy-|\det(B)|\right|\leq \epsilon |\det(B)|$. Thus,
$$ 
 \int_{\R^n}\|Tf_1(y)\|\, dy= \int_{\R^n}|g(y)|\, dy=\det(B).
$$
Further,
$$
\int_{\R^n}f(y)^\alpha dy=\frac{1}{\det(B)}\int_{\R^n}f(B^{-1}y)^\alpha dy=\frac{1}{|\det(B)}\int_{\R^n}|g(y)dy=1=\int_{\R^n}f(y)dy.
$$
Since, $f(y)-f(y)^\alpha$ has same sign on $\R^n$, it follows that $\alpha=1$.
This completes the proof of Theorem \ref{th: main th r 1}.
\end{proof}

\begin{remark}
In step 7, using the hypothesis $T(S_H)=S_H$ and the bijectivity of $T$, we have shown that $\beta:\C\to\C$ is bijective and continuous. These properties gave the particular form of $\beta$. Here we would like to mention that if additionally we assume that $T$ is additive, then one can show that $\beta$ is an additive and multiplicative bijection on $\C$, which leaves $\R$ invariant. As a consequence of these properties we would easily get $\beta(z)=z$ or $\beta(z)=\bar z.$
\end{remark}

\begin{remark}
    In step 8 we have defined the function $\beta$ using $e_1$. However, it turns out that $\beta(c)=\langle T(ce_k),e_k\rangle$, for any $k\in\N$. In fact,
    $$
    \beta(c)=\langle T(ce_1),e_1\rangle= \left\langle\langle T(ce_1), Te_1\rangle Te_k,e_k\right\rangle=\left\langle T\left(\langle ce_1,e_1\rangle e_k\right), e_k\right\rangle=\langle T(ce_k),e_k\rangle,
    $$
    using $Te_k=e_k$ and \eqref{eq: u lin 3.1}.
\end{remark}

We now obtain the characterization of the Fourier transform on $S_H$ as a consequence of Theorem \ref{th: main th r 1}.
\begin{proof}[Proof of Theorem \ref{th: main th r}]
    Let $ T=U\circ\F$. Then, it is easy to see that $T$ satisfies the hypotheses of Theorem \ref{th: main th r 1}. Thus appealing to Theorem \ref{th: main th r 1}, we obtain the desired characterization of the Fourier transform.
\end{proof}

\begin{corollary}\label{cor: main th cor}
    If in addition to the hypothesis of Theorem \ref{th: main th r}, $U$ satisfies $U(g_k^\varepsilon)(\xi)=g_k(\varepsilon\xi)$, for all $k\in \N$ and $\varepsilon>0$, where $g_k(x)=e^{-\pi |x|^2}e_k$ , and $g_k^\varepsilon(x)=\varepsilon^{-n}g_k(x/\varepsilon)$ then,
    \begin{itemize}
        \item [(i)] either, $Uf(\omega)= \F f(B\omega)$ for all $f\in S_H$ and $\omega\in \R^n,$ 
        \item [(ii)] or, $Uf(\omega)= \displaystyle\sum_k\overline{\langle \F f(B\omega), e_k \rangle}e_k$ for all $f\in S_H$ and $\omega\in \R^n$,
        \end{itemize}
    and $\det(B)=\pm 1$.
\end{corollary}
\begin{proof}
We apply Theorem \ref{th: main th r} with $g=g_k^{\varepsilon}$ and let $\varepsilon$ tend to zero, to arrive at (i), (ii).  In order to see that $|\det(B)|=1,$ observe that $\F(g_1^1)=g_1^1$. Thus appealing to Theorem \ref{th: main th r}, we obtain
$$ \left\|g_1^1(\xi)\right\|=\left\|Ug_1^1(\xi)\right\|=\left\|\F(g_1^1)(B\xi)\right\|=\left\|g_1^1(B\xi)\right\|.$$
Now, integrating both sides, it follows that $|\det(B)|=1.$

\end{proof}

\begin{corollary}
Let $H$ be a separable Hilbert space with an orthonormal basis $\{e_k:k\in\N\}$. Suppose that $U: S_H \to S_H$ is a bijective mapping with bijective extension (which we denote by $U$ again) on $S_H^\prime$. Assume that $U$ satisfies the following:
    \begin{itemize}
        \item [(i)] If $f$ is orthogonal to $e_k$, then $Uf$ is orthogonal to $e_k$,  for all $f\in S_H^\prime$ and $U(e_k)\in H$,
     \item[(ii)] $U\widetilde{f}(x)=\widetilde{Uf}(x),$ where $\widetilde{f}(x)=f(-x)$ for all $f \in S_H,$ 
        %\item [(i)] If $f$ is orthogonal to $e_k$, then $Tf$ is orthogonal to $e_k$, for all $f\in S_H^\prime$.
       % \item[(iii)] $U\left(f \right)=\sum_{k\in\N} U\left(\langle f(\cdot),e_k\rangle e_k\right)$, for all $f\in S_H$,
        \item[(iii)] For every $f \in S_H'$, $\langle U^2f, \phi \rangle= \langle f, \widetilde {\phi}\rangle$ for all $\phi \in S_H$.
    \end{itemize} And that one of the following two identities holds:
    \begin{itemize}
        \item [(1)]  $U((f\fisheye g)\circledast h)=(Uf \blackinwhitesquare Ug) \otimes Uh)$ $\forall f \in S_H'$ and $\forall g,h \in S_H,$
        \item  [(2)] $U((f\blackinwhitesquare g)\otimes h)=(Uf\fisheye Ug)\circledast Uh)$ $\forall f \in S_H'$ and $\forall g,h \in S_H.$
    \end{itemize}
    Then there exists $B \in M_n(\R)$ such that 
    $$\left\langle Uf(\omega), Ug(\omega)\right\rangle = \left\langle \F f(B\omega), \F g(B\omega) \right\rangle\,\,\text{ or }\,\, \left\langle Uf(\omega), Ug(\omega)\right\rangle = \left\langle \F g(B\omega), \F f(B\omega) \right\rangle, \,\,\, f,g \in S_H, \omega\in \R^n.$$
\end{corollary}

\begin{proof}
    If for every $f \in S_H'$, $\langle U^2f, \phi \rangle= \langle f, \widetilde {\phi}\rangle$, then if one of the last two identities holds, so does the second one, so Theorem $1.7$ gives the result. 
\end{proof}

\vspace{10pt}

\section{Characterization of Fourier transform of $H$-valued functions on $\Z_n$}
In this section we aim to characterize the Fourier transform on $L^1(\Z/n\Z, H)$. In order to do so, let us fix some notation first. 
%\subsection{Definitions}
For $v\in H$, we denote the element $(v,\cdots,v)$ in $L^1(\Z_n, H)$ by $\mathbb{1}_v$. For notational convenience, we write $\mathbb{1}_k$ for $\mathbb{1}_{e_k}$. For $j\in\Z_n,\,k\in\N$, $\delta_j^k$ will denote the element $\delta_je_k$.
\par Recall that the Fourier transform of $a\in L^1(\Z_n,H)$ is given by
\[
\F(a)(\ell)=\sum_{j=0}^{n-1} e^{-\frac{2\pi ij\ell}{n}}a(j),\,\, \ell\in\Z_n.
\]
A straightforward computation shows that $\F^2(a)(\ell)= na(-\ell)$. For $a,b\in L^1(\Z_n,H)$, we define the convolution of $a,b$, analogous to the continuous case, as follows.
\[
a\ast b(\ell)=\displaystyle\sum_{j=0}^{n-1}\langle a(j), b(\ell-j) \rangle,\,\,\ell\in\Z_n
\]
Then it is easy to see that $\F(a\ast b)= \F a \times \F b,$ where $(a\times b)(\ell):=\langle a(\ell),b(-\ell)\rangle$. Moreover, for $a\in L^1(\Z_n),\, b\in L^1(\Z_n,H)$, if we define 
\[
(a\mdsmblkcircle b)(\ell):=\displaystyle\sum_{j=0}^{n-1} a(j) b(\ell-j) ,~~\ell\in\Z_n,
\]
then it follows that $\F(a\mdsmblkcircle b)= \F(a)\smblksquare \F(b)$, where $(a\smblksquare b)(\ell)=a(\ell)b(\ell)$. Combining these we arrive at the following
\begin{theorem}\label{th: dis conv to pro}
    If $a,b,c \in L^1(\Z_n,H)$, then $(a\ast b)\mdsmblkcircle c \in L^1(\Z_n, H)$ and $\F ((a\ast b)\mdsmblkcircle c)= (\F a\times \F b)\smblksquare \F c.$
\end{theorem}
Further, by making use of $\F^2(a)(\ell)= n a(-\ell)$ together with Theorem \ref{th: dis conv to pro}, we obtain the following.
\begin{theorem}
    Let $a,b,c \in L^1(\Z_n,H)$. Then $(a\times b)\smblksquare c \in L^1(\Z_n, H)$ and $\F ((a\times b)\smblksquare c)= \frac{1}{n^2}(\F a\ast \F b)\mdsmblkcircle \F c.$
\end{theorem}

% \begin{proof}
% Consider,
%         \begin{align*}
%         \F((\F a \times \F b)\smblksquare \F c)&= \F(\F ((a\ast b)\mdsmblkcircle c))\\
%         &= \F^2((a\ast b)\mdsmblkcircle c)\\
%         &= \widetilde{((a\ast b)\mdsmblkcircle c)}\\
%         &= ((\widetilde {a}\ast \widetilde {b})\mdsmblkcircle \widetilde {c})\\
%         &= ((\F^2 a\ast \F^2 b)\mdsmblkcircle \F^2 c)\\
%         &= ((\F (\F a)\ast \F (\F b))\mdsmblkcircle \F (\F c)),
%     \end{align*}
% using $\F^2(a)= \widetilde {a}$, where $\widetilde {a}(k)=a(-k).$ Now using the bijectivity of $\F$, we get the desired result.
% \end{proof}

We now show that above convolution properties characterize the Fourier transform on $L^1(\Z_n,H)$.
\begin{theorem}\label{th: main discrete}
    Let $U: L^1(\Z_n, H) \to L^1(\Z_n,H)$ be a bijective transformation (not necessarily linear) such that the map $c\mapsto T(\mathbb{1}_{ce_k})$ is continuous, for any $c\in\C$ and  $T(\mathbb{1}_k)$ is a constant. Assume that
    for all $a,b,c \in L^1(\Z_n,H)$ 
    \begin{itemize}
        \item [(i)] $U((a\times b)\smblksquare c)= \frac{1}{n^2}(U(a)\ast U(b))\mdsmblkcircle U(c).$
        \smallskip
        
        \item [(ii)] $U((a\ast b)\mdsmblkcircle c)=(U(a)\times U(b))\smblksquare U(c).$
        \smallskip
        
 		%\item[(iii)] $U(a)=\sum_{k\in\N}U(\langle a,e_k\rangle e_k)$, where $a$ has the pointwise orthonormal expansion $a_j=\sum_{k\in\N}\langle a_j,e_k\rangle e_k$.
        \smallskip
        
        \item[(iii)] If $a\perp e_k$ for some $k\in\N$, then $U(a)\perp e_k$. 
        % \item[(iv)] The map $c\mapsto T(\mathbb{1}_{ce_k})$ is continuous, for any $c\in\C$ and  $T(\mathbb{1}_k)$ is a constant.
    \end{itemize}
    Then there exists $\eta \in \{1,2, \cdots  n-1\}$ that has no common divisor with $n$ such that
    \[
    \langle T(a)(\eta j), T(b)(\eta j)\rangle= \langle \F(a)(j), \F(b)(j)\rangle~~\text{ or }\, \langle T(a)(\eta j), T(b)(\eta j)\rangle= \langle \F(b)(j), \F(a)(j)\rangle,~~j \in \Z_n, a,b \in L^1(\Z_n,H).
    \]
    % \begin{itemize}
    %     \item [-] either, for each $j \in \Z_n$, for each $a,b \in L^1(\Z_n,H)$, $\langle T(a)(\eta j), T(b)(\eta j)\rangle= \langle a(j), b(j)\rangle$,
    %     \item [-] or, for each $j \in \Z_n$, for each $a,b \in L^1(\Z_n,H)$, $\langle T(a)(\eta j), T(b)(\eta j)\rangle= \overline{\langle a(j), b(j)\rangle}$.
    % \end{itemize}
\end{theorem}
As in the case of functions on $\R^n$, it is enough to prove following 

\begin{theorem} \label{T4.4}
    Let $T: L^1(\Z_n, H) \to L^1(\Z_n,H)$ be a bijective transformation (not necessarily linear) such that the map $c\mapsto T(\mathbb{1}_{ce_k})$ is continuous, for any $c\in\C$ and  $T(\mathbb{1}_k)$ is a constant. Assume that
    for all $a,b,c \in L^1(\Z_n,H)$ 
    \begin{itemize}
        \item [(i)] $T((a\times b)\smblksquare c)= (T(a)\times T(b))\smblksquare T(c).$
        \item [(ii)] $T((a\ast b)\mdsmblkcircle c)=(T(a)\ast T(b))\mdsmblkcircle T(c).$
 		%\item[(iii)] $T(\mathbb{1}_k)$ is a constant. 
        \item[(iii)] If $a\perp e_k$ for some $k\in\N$, then $T(a)\perp e_k$.
        % \item[(iv)] The map $c\mapsto T(\mathbb{1}_{ce_k})$ is continuous, for any $c\in\C$ and  $T(\mathbb{1}_k)$ is a constant.
    \end{itemize}
    Then there exists $\eta \in \{1,2, \cdots  n-1\}$ that has no common divisor with $n$ such that
    \[
    \langle T(a)(\eta j), T(b)(\eta j)\rangle= \langle a(j), b(j)\rangle~~\text{ or }\, \langle T(a)(\eta j), T(b)(\eta j)\rangle= \langle b(j), a(j)\rangle,~~j \in \Z_n, a,b \in L^1(\Z_n,H).
    \]
    % \begin{itemize}
    %     \item [-] either, for each $j \in \Z_n$, for each $a,b \in L^1(\Z_n,H)$, $\langle T(a)(\eta j), T(b)(\eta j)\rangle= \langle a(j), b(j)\rangle$,
    %     \item [-] or, for each $j \in \Z_n$, for each $a,b \in L^1(\Z_n,H)$, $\langle T(a)(\eta j), T(b)(\eta j)\rangle= \overline{\langle a(j), b(j)\rangle}$.
    % \end{itemize}
\end{theorem}
\vspace{10pt}

\begin{proof}
As in the case of $\R^n$ in the previous section, we prove the theorem in the following steps.

\noindent
%\begin{enumerate}
{\textbf{Step 1 : }}
In this step, we investigate the action of $T$ on $0$ and $\mathbb{1}_k$. Appealing to the hypothesis (iii), it is easy to see that $T(0)=0$.
\par By making use of the hypothesis (iii) and the continuity of $T$ it follows that $T(\mathbb{1}_k)(j)= ce_k$ for some $c\in\C$. Observe that, if $a\in L^1(\Z_n,H)$ is such that $a\perp e_l$, for all $l\neq k$, then $(a\times \mathbb{1}_k)\smblksquare \mathbb{1}_k= a$. Applying $T$ both sides we obtain, $(T(a)\times T(\mathbb{1}_k))\smblksquare T(\mathbb{1}_k)= T(a)$. Then appealing to the bijectivity of $T$ and the continuity of $T$, it follows that $(b\times T(\mathbb{1}_k))\smblksquare T(\mathbb{1}_k)=b$ for all $b\in L^1(\Z_n,H)$ with $b\perp e_l$ for all $l\neq k.$ Now, taking $b= \mathbb{1}_k$, an easy calculation shows that $|c|=1$. As in the case of $\R^n$, for simplicity, we proceed with $c=1.$

\vspace{10pt}

 \noindent
{\bf Step 2:} In this step we show that the Support of $T\delta_j^k$ is singleton. Suppose that $Supp(T\delta_j^k)$ contains two distinct points $m,\, l.$ Then $T(\delta_j^k)(r) \neq 0$ for $r= m,l$. Let $a=T^{-1}(\delta_m^k), b= T^{-1}(\delta_l^k)$. Then

\begin{align*}
    (a\times \mathbb{1}_k)\smblksquare \delta_j^k = (T^{-1}(\delta_m^k)\times T^{-1}(\mathbb{1}_k))\smblksquare T^{-1}(T(\delta_j^k))= T^{-1}((\delta_m^k\times \mathbb{1}_k)\smblksquare T(\delta_j^k))\neq T^{-1}(0) 
    = 0,
\end{align*}
since
\begin{align*}
    (\delta_m^k\times \mathbb{1}_k)\smblksquare T(\delta_j^k)(m)= \langle \delta_m^k(m), \mathbb{1}_k(m)\rangle T\delta_j^k(m)= \langle e_k, e_k \rangle T\delta_j^k(m)\neq 0.
\end{align*}
Hence, $(a\times \mathbb{1}_k)\smblksquare \delta_j^k \neq 0$. Similarly, $(b\times \mathbb{1}_k)\smblksquare \delta_j^k \neq 0$, which implies that $(a\times \mathbb{1}_k)\smblksquare b \neq 0.$
% \begin{itemize}
%     \item $(a\times \mathbb{1}_k)\smblksquare \delta_j^k \neq 0$
%     \item $(b\times \mathbb{1}_k)\smblksquare \delta_j^k \neq 0$
% \end{itemize}
%From these two, we can conclude that, $(a\times \mathbb{1}_k)\smblksquare b \neq 0.$

On the other hand, for $r\in\Z_n$, we have \begin{align*}
    (T(a)\times T(\mathbb{1}_k))\smblksquare T(b)(r) = (\delta_m^k \times \mathbb{1}_k)\smblksquare \delta_l^k(r)= \langle \delta_m^k(r), e_k \rangle \delta_l^k (r)= 0, 
    %&& \text{since, } m \neq l
\end{align*}
which implies that $(a\times \mathbb{1}_k)\smblksquare b=0$. This is a contradiction. Hence, $Supp(T\delta_j^k)$ is singleton.

\vspace{10pt}

\noindent
{\bf Step 3:} In this step, we find the image of $\delta_0^k$ under $T$. Since, $Supp (T \delta_0^k)$ is singleton, $T\delta_0^k= c\delta_{u(0)}^k$ for some $c\in \C$ and $u(0)\in \Z_n.$ We show that $|c|=1$ and $u(0)=0$. 

\par For $a\in L^1(\Z_n,H)$ with $a\perp e_l$ for all $l\neq k$, one can show that $\left( a\ast \delta_0^k\right)\mdsmblkcircle \delta_0^k= a$.
Applying $T$ on both sides we obtain, $\left(Ta\ast T\delta_0^k\right)\mdsmblkcircle T\delta_0^k= Ta$. Then by making use of the bijectivity of $T$ and the hypothesis (iii), it follows that $\left(b\ast T\delta_0^k\right)\mdsmblkcircle T\delta_0^k= b$ for all $b\in L^1(\Z_n,H)$ with $b\perp e_l$ for all $l\neq k.$ Now, taking $b=\delta_0^k$, we get $\left(\delta_0^k \ast c\delta_{u(0)}^k\right)\mdsmblkcircle c\delta_{u(0)}^k= \delta_0^k.$ Since, $\delta_0^k\perp e_\ell$, for all $\ell\neq k$, we have $|c|^2\delta_{2u(0)}^k= \delta_0^k$, which immediately gives $u(0)=0$ and $|c|=1$.
For simplicity, we proceed with $c=1.$

\vspace{10pt}

\noindent
{\bf Step 4:} In this step we show that there is an $\eta \in \{1, \cdots, n-1\}$ having no common divisor with $n$ such that $T(\delta_j^k)= \delta_{\eta j}^k$.

Appealing to the previous step, we have $T(\delta_j^k)= \delta_{\sigma(j)}^v$, for some $v \in H$, $\sigma(j) \in \{0,1, \cdots , n-1\}.$ Since, $\delta_j^k \perp e_l$ for all $l \neq k,$ our hypothesis implies that $T(\delta_j^k) \perp e_l$ for all $l \neq k.$ Hence $\langle \delta_{\sigma(j)}^v(x), e_l \rangle=0$ for all $l \neq k,$ which leads to $\langle v, e_l\rangle=0$ for all $l\neq k$. Hence $v=c_je_k$, for some $c_j\in\C$. Therefore we have, $T(\delta_j^k)= \delta_{\sigma(j)}^{c_je_k}= c_j \delta_{\sigma(j)}^k.$ We now show that $c_j=1.$ Observe that
\begin{align*}
    (\mathbb{1}_k\ast \delta_0^k)\mdsmblkcircle \delta_j^k(l)
    %&= \sum_{s=0}^{n-1}(\mathbb{1}_k\ast \mathbb{1}_k)(s)\delta_j^k(l-s)\\
    = \sum_{s=0}^{n-1}\left(\sum_{t=0}^{n-1}\langle \mathbb{1}_k(t), \delta_0^k(s-t)\rangle\right)\delta_j^k(l-s)
    = \sum_{s=0}^{n-1}  \delta_j^k(l-s)
    %&= n \sum_{s=0}^{n-1}\delta_j^k(l-s)\\
    = e_k
    = \mathbb{1}_k(l).
\end{align*}
Thus, 
%using the additivity of $T$, we get
\begin{align*}
    \mathbb{1}_k = T( \mathbb{1}_k)= T((\mathbb{1}_k\ast \delta_0^k)\mdsmblkcircle \delta_j^k)= (\mathbb{1}_k \ast \delta_0^k)\mdsmblkcircle \delta_{\sigma(j)}^k= c_j  \mathbb{1}_k.
\end{align*}
This shows that $c_j=1.$ Hence, $T(\delta_j^k)= \delta_{\sigma(j)}^k.$
Using the bijectivity of $T$, it follows that $\sigma(j)$ is a permutation on $\{1,\cdots ,n-1\}$.\\
In order to show that $\sigma(j)$ does not have a common divisor with n, consider
%Now, we claim that $(\delta_i^k\ast \delta_j^k)\mdsmblkcircle \delta_0^k= \delta_{i+j}^k.$
\begin{align*}
    (\delta_l^k\ast \delta_m^k)\mdsmblkcircle \delta_0^k(j) 
    %= \sum_{s=0}^{n-1}(\delta_i^k\ast \delta_j^k)(s)\delta_0^k(l-s)\\
    = \sum_{s=0}^{n-1}(\sum_{t=0}^{n-1}\langle \delta_l^k(t), \delta_m^k(s-t)\rangle)\delta_0^k(j-s)= \sum_{t=0}^{n-1}\langle \delta_l^k(t), \delta_m^k(j-t)\rangle e_k
    %&= \langle e_k, e_k \rangle e_k && \text{when} \, l=i+j,\\
    %&= 0 && \text{o/w}\\
    =\delta_{l+m}^k(j).
\end{align*}
Thus,
\begin{align*}
    \delta_{\sigma(l+m)}^k = T(\delta_{l+m}^k)= T((\delta_l^k\ast \delta_m^k)\mdsmblkcircle \delta_0^k)= (T(\delta_l^k)\ast T(\delta_m^k))\mdsmblkcircle T(\delta_0^k)= (\delta_{\sigma(l)}^k \ast \delta_{\sigma(m)}^k)\mdsmblkcircle \delta_0^k= \delta_{\sigma(l)+\sigma(m)}^k,
\end{align*}
shows that $\sigma(l+m)= \sigma(l)+\sigma(m).$ This implies that $\sigma(m)= m \sigma(1).$ Since, $\sigma$ is a permutation, appealing to B\`ezout's theorem we get that $\sigma(1)$ has no common divisor with $n$.

\vspace{10pt}

\noindent
{\bf Step 5:} Finally, in this step we obtain the desired characterization. First, we derive some useful identities. For $a \in L^1(\Z_n, H)$, we let $E[a]:= \sum_{j=0}^{n-1}a_j$. Then for $j\in \Z_n$,
    \begin{align*}
        (a\ast \mathbb{1}_k)\mdsmblkcircle \delta_0^k(j)%&= \sum_{s=0}^{n-1}(a\ast \mathbb{1}_k)(s)\mathbb{1}_k(i-s)\\
        =\sum_{s=0}^{n-1}\left(\sum_{t=0}^{n-1}\left\langle a(t), \mathbb{1}_k(s-t)\right\rangle\right)\delta_0^k(j-s)
        %&= \sum_{s=0}^{n-1}(\sum_{t=0}^{n-1}\langle a(t), e_k\rangle)e_k\\
        =\sum_{s=0}^{n-1}\left\langle \sum_{t=0}^{n-1} a(t), e_k \right\rangle e_k
        %&= \sum_{s=0}^{n-1}\langle E[a], e_k \rangle e_k\\
        &= \langle E[a], e_k \rangle e_k\\
        &= (\mathbb{1}_{E[a]}\times \mathbb{1}_k)\smblksquare \mathbb{1}_k(j).
    \end{align*}
    Hence $(a\ast \mathbb{1}_k)\mdsmblkcircle \delta_0^k= (\mathbb{1}_{E[a]}\times \mathbb{1}_k)\smblksquare \mathbb{1}_k.$ Now applying $T$ on both the sides we obtain, $(Ta\ast \mathbb{1}_k)\mdsmblkcircle \delta_0^k= \left(T(\mathbb{1}_{E[a]})\times \mathbb{1}_k\right)\smblksquare \mathbb{1}_k$, which leads to %$n(T(\mathbb{1}_{E[Ta]})\times \mathbb{1}_k)\smblksquare \mathbb{1}_k= n(T(\mathbb{1}_{E[a]})\times \mathbb{1}_k)\smblksquare \mathbb{1}_k$. So we have 
    $$(\mathbb{1}_{E[Ta]}\times \mathbb{1}_k)\smblksquare \mathbb{1}_k= (T(\mathbb{1}_{E[a]})\times \mathbb{1}_k)\smblksquare \mathbb{1}_k.$$

    We now define the map $\beta:\C\to\C$, by $\beta(c):= \left\langle E\left[T\left((c/n)\mathbb{1}_k\right)\right], e_k \right\rangle$,\, $c \in \C.$ Then, for $j\in\Z_n,$
    \begin{align*}
        \beta(c)\mathbb{1}_k(j) 
        %&= \langle E[T(c/n)\mathbb{1}_k], e_k \rangle e_k \\
        &= \left\langle \mathbb{1}_{E\left[T\left((c/n) \mathbb{1}_k\right)\right]}(j), \mathbb{1}_k(j)\right\rangle \mathbb{1}_k(j)
        = (\mathbb{1}_{E[T(c/n) \mathbb{1}_k]}\times \mathbb{1}_k)\smblksquare \mathbb{1}_k(j)
        = (T((\mathbb{1}_{E[(c/n)\mathbb{1}_k]})\times \mathbb{1}_k)\smblksquare \mathbb{1}_k)(j)
        = T(\mathbb{1}_{ce_k})(j).
    \end{align*}
    %So, $\beta(c)\mathbb{1}_k(j)= T(\mathbb{1}_{ce_k})(j)$. 
    Hence we have the following identity,
    $$\beta(c)\mathbb{1}_k= T(\mathbb{1}_{c e_k}).$$

    Moreover, for $u,v\in H$, we have 
    \begin{align*}
        \beta(\langle u, v \rangle)\mathbb{1}_k = T(\mathbb{1}_{\langle u,v \rangle e_k})= T((\mathbb{1}_u\times \mathbb{1}_v)\smblksquare \mathbb{1}_k)= (T(\mathbb{1}_u)\times T(\mathbb{1}_v))\smblksquare \mathbb{1}_k.
    \end{align*}
    So, $$\beta(\langle u, v \rangle)= \langle T(\mathbb{1}_u)(j), T(\mathbb{1}_v)(j)\rangle,$$ for any $j \in \Z_n.$ Now, for $a, b \in L^1(\Z_n, H)$ and $j \in \Z_n$, we have
\begin{align*}
    \langle T(a)(\eta j), T(b)(\eta j) \rangle\delta_{\eta j}^k &= (T(a)\times T(b))\smblksquare \delta_{\eta j}^k\\
        &= (T(a)\times T(b))\smblksquare T(\delta_j^k)\\
        &= T((a\times b)\smblksquare \delta_ j^k)\\
        &= T(\langle a(j), b(j)\rangle \delta_j^k)\\
        &= T((\mathbb{1}_{a(j)}\times \mathbb{1}_{b(j)})\smblksquare \delta_j^k)\\
        %&= (T(\mathbb{1}_{a(j)})\times T(\mathbb{1}_{b(j)})\smblksquare T\delta_j^k\\
        &= \langle T(\mathbb{1}_{a(j)})(\eta j), T(\mathbb{1}_{b(j)})(\eta j)\rangle \delta_{\eta j}^k\\
        &= \beta(\langle a(j), b(j)\rangle)\delta_{\eta j}^k.
    \end{align*}
    Hence, 
    \begin{equation}\label{eq: zn form}
    \langle T(a)(\eta j), T(b)(\eta j)\rangle = \beta(\langle a(j), b(j)\rangle).
    \end{equation}
    Let the permutation $j\mapsto \eta j$ be denoted by $\sigma$. Then we rewrite the above equality as
\begin{equation}\label{eq: zn form 1}
\langle T(a)(j), T(b)(j)\rangle= \beta\left(\left\langle a(\sigma^{-1}(j)), b(\sigma^{-1}(j))\right\rangle\right).
\end{equation}
As in the case of $\R^n$, we show that $\beta$ is a continuous multiplicative bijection. The continuity of $\beta$ follows from the hypothesis (v).

    % 
    %     \langle T(a)(j), T(b)(j)\rangle= \beta\left(\left\langle a\left(\sigma^{-1}\left(j\right)\right), b\left(\sigma^{-1}\left(j\right)\right)\right\rangle\right).
    % \end{equation}
    
    %where $T(\delta_i^k)= \delta_j^k.$

    % As in the continuous case, the result follows once we show that
    %     $\beta$ is a additive and
    %      multiplicative
    %     bijection which maps $\R$ onto $\R$.
    %\end{enumerate}
For $c_1,c_2\in\C$, we have
    \begin{align*}
        \beta(c_1 c_2)\mathbb{1}_k = T(\mathbb{1}_{c_1 c_2 e_k})= T((\mathbb{1}_{c_1 e_k}\times \mathbb{1}_k)\smblksquare \mathbb{1}_{c_2 e_k})= (T(\mathbb{1}_{c_1 e_k})\times \mathbb{1}_k)\smblksquare T(\mathbb{1}_{c_2 e_k})
        &= (\beta(c_1)\mathbb{1}_k\times \mathbb{1}_k)\smblksquare \beta(c_2)\mathbb{1}_k\\
        %&= \beta(c_1)\beta(c_2)(\mathbb{1}_k\times \mathbb{1}_k)\smblksquare \mathbb{1}_k
        &= \beta(c_1)\beta(c_2)\mathbb{1}_k,
    \end{align*}
    which shows that $\beta$ is multiplicative. 
 %    The additivity of $\beta$ follows from

 % \begin{align*}
 %        \beta(c_1+c_2)\mathbb{1}_k = T(\mathbb{1}_{(c_1+c_2)e_k})= T(\mathbb{1}_{c_1e_k}+\mathbb{1}_{c_2e_k})= T(\mathbb{1}_{c_1e_k})+T(\mathbb{1}_{c_2e_k})= \beta(c_1)\mathbb{1}_k +\beta(c_2)\mathbb{1}_k= (\beta(c_1)+\beta(c_2))\mathbb{1}_k.
 %    \end{align*}
    %Hence, $\beta(c_1+c_2)= \beta(c_1)+\beta(c_2).$

%\vspace{10pt}
    
    %\textbf{$\beta$ is injective: }
    %Let, $\beta(c_1)= \beta(c_2)$, implies, $\beta(c_1)\mathbb{1}_k= \beta(c_2)\mathbb{1}_k$, implies, $T(\mathbb{1}_{c_1e_k})= T(\mathbb{1}_{c_2e_k})$, implies, $\mathbb{1}_{c_1 e_k}= \mathbb{1}_{c_2 e_k},$ implies, $c_1= c_2.$
    Using the injectivity of $T$, it is easy to see that $\beta$ is injective.
    %\textbf{$\beta$ is onto: }
    %Let $d \in \C.$ This implies $\mathbb{1}_{de_k}\in L^1(\Z_n, H)$. Now as $T$ is onto, so there exist $a \in L^1(\Z_n, H)$ such that $T(a)= \mathbb{1}_{de_k}.$ Now if we can show that $a= \mathbb{1}_{ce_k}$ for some $c\in \C$, then we are done. Now $T^{-1}(\mathbb{1}_{de_k})=a.$ $T^{-1}$ satisfies same property with $T$. Now $\mathbb{1}_{de_k} \perp e_l$ for all $\ \neq k,$ implies, $T^{-1}(\mathbb{1}_{de_k}) \perp e_l$ for all $l \neq k$, implies, $\langle a(j), e_l \rangle = 0$ for all $ l \neq k$, for all $j \in \Z_n,$ implies $a(j)= c_je_k$. Now we want to show that $a(j)=ce_k$. For that, our claim is that $T$ maps constant function to constant function, So, $T^{-1}$ will satisfy the same.
In order to show that $\beta$ is onto, first we show that $T$ maps constants to constants. For $v \in H$, we shall show that $T(\mathbb{1}_v)= \mathbb{1}_w,$ for some $w \in H.$ Let $g_v=\frac{1}{n^2}\mathbb{1}_v.$ Then for $i\in\Z_n$, we have
    \begin{align*}
        (\mathbb{1}_k\ast \mathbb{1}_k)\mdsmblkcircle g_v(j) 
        %&= \sum_{s=0}^{n-1}(\mathbb{1}_k\ast f_k)(s)g_v(i-s)\\
        &=  \sum_{s=0}^{n-1}\left(\sum_{t=0}^{n-1}\langle \mathbb{1}_k(t), \mathbb{1}_k(s-t)\rangle\right)g_v(j-s)
        = \sum_{s=0}^{n-1}\sum_{t=0}^{n-1}\langle e_k, e_k \rangle g_v(j-s)
        %&=  \sum_{s=0}^{n-1} (n.\frac{1}{n})\frac{1}{n}v\\
        = v
        = \mathbb{1}_v(j).
    \end{align*}
%So, we have $(\mathbb{1}_k\ast f_k)\mdsmblkcircle g_v= \mathbb{1}_v$. Now applying $T$ on both sides we have, 
Thus, $T(\mathbb{1}_v)= (\mathbb{1}_k\ast \mathbb{1}_k)\mdsmblkcircle T(g_v)$. Moreover,
    \begin{align*}
        (\mathbb{1}_k\ast \mathbb{1}_k)\mdsmblkcircle T(g_v)(j) %&= \sum_{s=0}^{n-1} (\mathbb{1}_k\ast Tf_k)(s)Tg_v(i-s)\\
        =\sum_{s=0}^{n-1}\sum_{t=0}^{n-1}\langle e_k, e_k\rangle Tg_v(j-s)
        %&= \sum_{s=0}^{n-1}(\sum_{t=0}^{n-1}\langle e_k, c_{s-t}e_k\rangle)Tg_v(i-s)\\
        %= \sum_{s=0}^{n-1}(\sum_{t=0}^{n-1}\overline{c_{s-t}})Tg_v(i-s)
        = n\sum_{s=0}^{n-1}Tg_v(j-s)   %&& \text{ where, } c= \sum_{t=0}^{n-1}\overline{c_{s-t}} \\
        = w,~\text{ a constant}.
        %&& \text{which is a constant}
    \end{align*}
%Hence we have, $T(\mathbb{1}_v)= \mathbb{1}_w.$
This shows that $T$ maps a constant function to a constant function. Similarly, we can show that $T^{-1}$ maps constant to constant. Let $d\in \C$. Then using ontoness of $T$, we have $a\in L^1(\Z_n,H)$ such that $T(a)=\mathbb{1}_{de_k}$. This is same as $T^{-1}(\mathbb{1}_{de_k})=a$. Since $T^{-1}$ maps constants to constants, $a=\mathbb{1}_v$, for some $v\in H$. Since, $\mathbb{1}_{de_k}\perp e_l$, for $l\neq k$, using our hypothesis, it follows that $a=\mathbb{1}_v=\mathbb{1}_{ce_k}$, for some $c\in\C$. Then it is easy to see that $\beta(c)=d$.
\par 
The above properties imply that $\beta$ has the following form.
$$\beta(c)=\left( \frac{c}{|c|}\right)^k|c|^\alpha,~~\, k \in \{\pm 1\},\, Re(\alpha)>0.$$
In order to obtain the desired equality, it is enough to show that $\alpha=1$.
\par By making use of \eqref{eq: zn form 1}, we have 
\begin{equation}\label{eq: disc T form}
\langle T(a)(j), T(b)(j)\rangle = \left(\frac{\left\langle a\left(\sigma^{-1}\left(j\right)\right), b\left(\sigma^{-1}\left(j\right)\right)\right\rangle}{\left|\left\langle a\left(\sigma^{-1}\left(j\right)\right), b\left(\sigma^{-1}\left(j\right)\right)\right\rangle\right|}\right)^k\left|\left\langle a\left(\sigma^{-1}\left(j\right)\right), b\left(\sigma^{-1}\left(j\right)\right)\right\rangle\right|^\alpha,
\end{equation}
for all $a,b\in L^1(\Z_n,H)$. For $t>0$, set $a_t= (e_1,\sqrt{t}e_1,0, \cdots  ,0)$. Then, using the hypothesis (iii), it follows that $T(a)(j)=c_je_1$ for some $c_j\in\C$. Let $\sigma^{-1}(j_0)=0$ and $\sigma^{-1}(j_1)=1$. Now, appealing to \eqref{eq: disc T form}, we get 
$$\left\langle T(a_t)(j), T(\mathbb{1}_k)(j)\right\rangle = \left(\frac{\left\langle a_t\left(\sigma^{-1}\left(j\right)\right), \mathbb{1}_k\right\rangle}{|\left\langle a_t\left(\sigma^{-1}\left(j\right)\right), \mathbb{1}_k\right\rangle|}\right)^k\left|\left\langle a_t\left(\sigma^{-1}\left(j\right)\right), \mathbb{1}_k\right\rangle\right|^\alpha.$$
It follows from the above equality that $c_{j_0}=1$, $c_{j_1}=t^{\alpha/2}$ and $c_j=0$ for $j\neq j_0,\,j_1$. Observe that $t^{Re(\alpha)}=\langle T(a_t)(j_1),T(a_t)(j_1)\rangle=\langle a_t(1),a_t(1)\rangle^\alpha=t^\alpha$, which forces $\alpha$ to be real. Further, by putting $a=b=\mathbb{1}_{E[a_t]}$ in \eqref{eq: disc T form}, we get $\left\|E[T(a_t)]\right\|^2= \left\|E[a_t]\right\|^{2\alpha}$, which is same as $\left(1+\sqrt{t}\right)^\alpha= 1+\left(\sqrt{t}\right)^\alpha$. Since the last equality holds for all $t>0$, we get $\alpha=1$. This completes the proof of Theorem \ref{T4.4}.

\end{proof}

We have the following corollary in this context.
\begin{corollary}
    If in addition to the hypothesis of Theorem \ref{th: main discrete}, U satisfies $U(b_k)=b_k$, for all $k\in\N$, where 
    %{\color{red}\st{$a_k(j)= e_k$ for all $j=0,1,\cdots,n$ and}} 
    $b_k=\mathbb{1}_k+\F \mathbb{1}_k$, then
    \begin{itemize}
        \item [(i)] either, $U a(\eta j)= \F a(j)$ for all $a \in L^1(\Z_n,H)$ and $j\in \{0,1,\cdots , n-1\}.$ 
        \item [(ii)] or, $U a(\eta j)= \displaystyle\sum_k \overline{\langle \F ( a)(j), e_k\rangle}e_k$ for all $a\in L^1(\Z_n,H)$ and $j\in \{0,1,\cdots ,n-1\}$.
    \end{itemize}

\end{corollary}
% \begin{remark}
% {\color{blue} The structure of the proof presented here closely follows that of \cite{Alesker_paper}, with significant adjustments, as we are working with two different convolutions and omitting technical details that are irrelevant in the context of finite groups. The principal distinction lies in the necessity of assuming that the operator TT acts continuously on constant functions, whereas \cite{Alesker_paper} assumes that $T$ maps smooth functions to smooth functions.

% It is also worth emphasizing that hypotheses (i) and (ii) are invoked solely in cases where either $a$ or $b$ is a constant $c_1$ or a Dirac measure.}
% \end{remark}

\subsection*{Remarks}
In \cite{Hasan1}, the authors have shown that $\mathcal{L}^2(\R,\mathcal{B})$, containing all $\mathcal{B}$-valued square (Bochner) integrable functions on $\R$ is a Hilbert $C^*$-module, where $\mathcal{B}$ is a Hilbert $C^*$-module. Hilbert $C^*$-modules generalizes the notion of a Hilbert space by allowing the inner product to take values in more general $C^*$-algebras rather than that of the complex numbers. They have shown that the Fourier transform is a unitary operator on $\mathcal{L}^2(\R,\mathcal{B})$. Moreover, viewing $\mathcal{L}^2(\R,\mathcal{B})$ as the tensor product of $L^2(\R)$ and $\mathcal{B}$, the authors have shown that an operator of the form $S\otimes I$ is the Fourier transform on $\mathcal{L}^2(\R,\mathcal{B})$ under some additional hypotheses.  
It is known that the Fourier transform maps $S_{\mathcal{B}}$ into $S_{\mathcal{B}}$ (see Proposition 2.4.22 in \cite{VVFT}), where $\mathcal{B}$ is a Banach space and $S_\mathcal{B}$ denotes the space of all $\mathcal{B}$-valued Schwartz class functions. So, it is natural to ask for a characterization of the Fourier transform on $S_{\mathcal{B}}$, where $\mathcal{B}$ is a Hilbert $C^*$-module or more generally a Banach space. 

\medskip
%\smallskip
%\bigskip
\noindent{\bf Acknowledgements.} 
This work is a part of first author's Ph. D. thesis. He thanks the National Institute of Science Education and Research Bhubaneswar, India, for providing an excellent research facility. This work started at NISER Bhubaneswar, India, during the second author's postdoctoral tenure. The second author is thankful to the National Board for Higher Mathematics (NBHM) under the Department of Atomic Energy (DAE), Govt. of India for the postdoctoral fellowship with reference number 0204/27/(14)/2023/R\&D-II/11891.
RM acknowledges the partial support provided by the research grants  (ANRF/ARGM/2025/000750/MTR).
%RM acknowledges the partial support provided by the research grants (DST/INSPIRE/04/2019/001914).

%\bigskip

\medskip

\bibliography{ref12}
\bibliographystyle{amsplain}
\end{document}